\documentclass[a4paper,12pt, oneside]{amsart}
\usepackage{mathrsfs}
\usepackage{amsfonts}
\usepackage{amssymb}
\usepackage[polish, english]{babel}
\usepackage{amsxtra}
\usepackage{color}
\usepackage{dsfont}
\usepackage[margin=2cm, centering]{geometry}
\usepackage[bbgreekl]{mathbbol}
\usepackage{soul}

\usepackage{array}
\usepackage{makecell}
\usepackage{multirow}

\usepackage[colorlinks,citecolor=blue,urlcolor=blue,bookmarks=true]{hyperref}

\DeclareSymbolFontAlphabet{\mathbb}{AMSb}
\DeclareSymbolFontAlphabet{\mathbbl}{bbold}

\usepackage{accents}

\theoremstyle{plain}
\newtheorem{theorem}{Theorem}[section] 
\newtheorem{lemma}[theorem]{Lemma}
\newtheorem{proposition}[theorem]{Proposition}
\newtheorem{corollary}[theorem]{Corollary}

 \theoremstyle{definition}

\newtheorem{example}{Example} 
\newtheorem*{remark*}{Remark} 
\newtheorem{remark}[theorem]{Remark}

\newcommand{\R}{\mathbb{R}}
\newcommand{\Rd}{{\R^{d}}}
\newcommand{\NN}{\mathbb{N}}
\newcommand{\ind}{\mathds{1}}
\renewcommand{\leq}{\leqslant}
\renewcommand{\geq}{\geqslant}
\newcommand{\Z}{\int^{\infty}_{0}}

\def \PP{\mathbb{P}}
\def \EE{\mathbb{E}}
\def \setA{F}

\def\({\left(} 
\def\){\right)} 
\def\[{\left[}
\def\]{\right]} 
\def\<{\langle} 
\def\>{\rangle}

\DeclareMathOperator{\supp}{supp}

\def \Ca{\rm{(C1)}}
\def \Cb{\rm{(C2)}}
\def \Cc{\rm{(C3)}}

\def \Ch{\rm{(C8)}}

\def \Aaa{\rm{(A{\!}^{'}1)}}
\def \Abb{\rm{(A{\!}^{'}2)}}
\def \Acc{\rm{(A{\!}^{'}3)}}

\def \Na{\rm{(C2)}}
\def \Nb{\rm{(E1)}}
\def \Nc{\rm{(E2)}}
\def \Nd{\rm{(C3)}}

\newcommand{\lah}{\alpha_h}

\newcommand{\LM}{N} 
\newcommand{\gLM}{n} 
\newcommand{\LCh}{\Psi} 
\newcommand{\lmC}{C_N} 
\newcommand{\drf}{b} 
\newcommand{\ubC}{C_{0}} 

\newcommand{\rr}{\Upsilon} 

\newcommand{\uLM}{\nu_0} 
\newcommand{\uLCh}{\LCh_0} 
\newcommand{\uh}{h_0} 
\newcommand{\uK}{K_0} 
\newcommand{\ulah}{\alpha_{h_0}} 

\definecolor{ks}{rgb}{0.7,0.1,0.2}

\definecolor{zm}{RGB}{255,0,255}

\allowdisplaybreaks

\title{Estimates of heat kernels of non-symmetric L{\'e}vy processes}
\thanks{The research was partially supported by
 the 
German Science Foundation (SFB 701)
and
National Science Centre (Poland)
grant 2016/23/B/ST1/01665.}

\author[T. Grzywny]{Tomasz Grzywny }
\address{
	Wydzia{\lll} Matematyki,
	Politechnika Wroc{\lll}awska\\
	Wyb. Wyspia\'{n}skiego 27\\
	50-370 Wroc{\lll}aw\\
	Poland
}
\email{tomasz.grzywny@pwr.edu.pl}

\author[K. Szczypkowski]{Karol Szczypkowski}
\email{karol.szczypkowski@pwr.edu.pl}



\date{}

\begin{document}


\def\1{{\bf 1}}
\def\ind{{\bf 1}}
\def\nn{\nonumber}
\newcommand{\I}{\mathbf{1}}

\def\sA {{\cal A}} \def\sB {{\cal B}} \def\sC {{\cal C}}
\def\sD {{\cal D}} \def\sE {{\cal E}} \def\sF {{\cal F}}
\def\sG {{\cal G}} \def\sH {{\cal H}} \def\sI {{\cal I}}
\def\sJ {{\cal J}} \def\sK {{\cal K}} \def\sL {{\cal L}}
\def\sM {{\cal M}} \def\sN {{\cal N}} \def\sO {{\cal O}}
\def\sP {{\cal P}} \def\sQ {{\cal Q}} \def\sR {{\cal R}}
\def\sS {{\cal S}} \def\sT {{\cal T}} \def\sU {{\cal U}}
\def\sV {{\cal V}} \def\sW {{\cal W}} \def\sX {{\cal X}}
\def\sY {{\cal Y}} \def\sZ {{\cal Z}}

\def\bA {{\mathbb A}} \def\bB {{\mathbb B}} \def\bC {{\mathbb C}}
\def\bD {{\mathbb D}} \def\bE {{\mathbb E}} \def\bF {{\mathbb F}}
\def\bG {{\mathbb G}} \def\bH {{\mathbb H}} \def\bI {{\mathbb I}}
\def\bJ {{\mathbb J}} \def\bK {{\mathbb K}} \def\bL {{\mathbb L}}
\def\bM {{\mathbb M}} \def\bN {{\mathbb N}} \def\bO {{\mathbb O}}
\def\bP {{\mathbb P}} \def\bQ {{\mathbb Q}} \def\bR {{\mathbb R}}
\def\bS {{\mathbb S}} \def\bT {{\mathbb T}} \def\bU {{\mathbb U}}
\def\bV {{\mathbb V}} \def\bW {{\mathbb W}} \def\bX {{\mathbb X}}
\def\bY {{\mathbb Y}} \def\bZ {{\mathbb Z}}
\def\R {{\mathbb R}} \def\RR {{\mathbb R}} \def\H {{\mathbb H}}
\def\n{{\bf n}} \def\Z {{\mathbb Z}}

\newcommand{\expr}[1]{\left( #1 \right)}
\newcommand{\cl}[1]{\overline{#1}}
\newtheorem{thm}{Theorem}[section]

\newtheorem{prop}[thm]{Proposition}

\numberwithin{equation}{section}
\def\ee{\varepsilon}
\def\qed{{\hfill $\Box$ \bigskip}}
\def\NN{{\mathcal N}}
\def\AA{{\mathcal A}}
\def\MM{{\mathcal M}}
\def\BB{{\mathcal B}}
\def\CC{{\mathcal C}}
\def\LL{{\mathcal L}}
\def\DD{{\mathcal D}}
\def\FF{{\mathcal F}}
\def\EE{{\mathcal E}}
\def\QQ{{\mathcal Q}}
\def\SS{{\mathcal S}}
\def\RR{{\mathbb R}}
\def\R{{\mathbb R}}
\def\L{{\bf L}}
\def\K{{\bf K}}
\def\S{{\bf S}}
\def\A{{\bf A}}
\def\E{{\mathbb E}}
\def\F{{\bf F}}
\def\P{{\mathbb P}}
\def\N{{\mathbb N}}
\def\eps{\varepsilon}
\def\wh{\widehat}
\def\wt{\widetilde}
\def\pf{\noindent{\bf Proof.} }
\def\pff{\noindent{\bf Proof} }
\def\cp{\mathrm{Cap}}

\begin{abstract}
We investigate 
densities of vaguely continuous convolution semigroups of probability measures on~$\Rd$.
First, we provide results that give upper estimates 
in a situation when the corresponding jump measure is allowed to be
highly non-symmetric.
Further, we prove upper 
estimates
of the 
density and its derivatives  if the jump measure compares
 with an isotropic unimodal measure and the characteristic exponent satisfies certain scaling condition.
Lower estimates are discussed  
in view of a recent development in that direction,
and in such a way to complement upper estimates.
We apply all those results to establish precise estimates of densities of non-symmetric 
L{\'e}vy processes.
\end{abstract}

\maketitle

\section{Introduction}

Heat kernels of stochastic processes have been studied since many years  
resulting in a large number of articles, where 
 various techniques and approaches
 were developed
(\cite{MR898496}, \cite{MR2492992}, \cite{MR1301283}, 
\cite{MR1402654}, \cite{TGBTMR}, \cite{MR2763092}, \cite{MR2357678}, \cite{MR2806700}, \cite{MR2886383}, \cite{MR3089797}, 
\cite{MR1744782}).
Among others,
heat kernels, also called densities, of  a rich class
of  L{\'e}vy processes were investigated
 (\cite{MR3139314}, \cite{MR2327834}, \cite{MR1829978}, \cite{MR3010850}, \cite{MR941977}, \cite{MR2925579}, \cite{MR2827465}, \cite{MR3666815}, \cite{MR3604626}, \cite{MR2995789}, \cite{MR2886459}).
A prominent example of a  L{\'e}vy process,
next to the Brownian motion with an explicit density,
 is
the isotropic $\alpha$-stable 
process
(\cite{MR0119247}, \cite{MR2274838}).
Sharp estimates of its density
stimulated the
analysis of
other subordinate Brownian motions (\cite{MR3570240}, \cite{MR2986850})
and even more general
isotropic unimodal L{\'e}vy processes (\cite{MR705619}, \cite{MR3165234}, \cite{MR3646773}).
For singular  $\alpha$-stable processes
or other (Lamperti, layered, relativistic, tempered, truncated) modifications and generalizations  see \cite{MR1085177}, 
\cite{MR1242643},
\cite {MR1220664},
 \cite{MR2320691}, \cite{MR2286060}, \cite{MR854867}, \cite{MR1829978}, \cite{MR2327834}, \cite{MR2307406}, \cite{MR2591907}, \cite{MR3604626}, \cite{MR2443765}, \cite{MR3413864}, \cite{MR1287843}, \cite{MR2025727}.
The list of directions and the
 literature above are most likely incomplete, thus
for more detailed descriptions 
we refer the reader to \cite{MR3139314}, \cite{MR2591907} (see also \cite{MR3604626}).

Most of the prior examples are symmetric L{\'e}vy processes, which initially attracted more attention.
For
 non-symmetric processes
one observes
 quite general 
but rather implicit estimates (\cite{MR3139314}, \cite{MR3235175}, \cite{MR3357585})
or studies  performed for very peculiar cases (\cite{MR0282413}, \cite{MR1486930}, \cite{MR3626900}, \cite{MR2794975}).
In that respect, our motivations to tackle the non-symmetric processes were twofold: propose a collection of general results that can lead to quite precise estimates in many interesting cases; identify a class of processes for which estimates can be treated 
in a unified manner.
Section~\ref{sec:gen},
~\ref{sec:upper_R}
and~\ref{sec:lower}
of the present paper
comprise those general results,
while
 in Section~\ref{sec:exls} we illustrate how to use them.
In Section~\ref{sec:cU}, on the other hand, we provide
estimates 
for a relatively large class of
non-symmetric
L{\'e}vy processes
in a legible and concise  form.

The
convolutional structure 
of L{\'e}vy processes
allows better understanding
and uniform results that can be beneficial in further studies.
They can, for instance, be employed to construct more complex
Markov processes. They usually serve as a starting point in that procedure, known as Levi's or parametrix method (\cite{MR3353627}, \cite{MR3652202}, \cite{MR3500272}, \cite{Kim2017}, \cite{MR2876511}, \cite{MR3550165}).
This was another motivation for the current study. The construction of  L{\'e}vy-type processes based on the results of Section~\ref{sec:cU} has already been accomplished in \cite{MR3996792} and \cite{KS-2018}.

For L{\'e}vy processes on the real line,
where the geometry is simpler,
the effects of the non-symmetry
appear to be more manageable.
For instance, in
\cite{MR4081943}, \cite{TGLLBT}, \cite{SCPK}
density estimates of certain subordinators, i.e., one-dimensional L{\'e}vy processes
supported on the positive half-line,
were established. 
In Section~\ref{sec:upper_R} 
we propose different general results
that allow handling highly non-symmetric one-dimensional L{\'e}vy processes.
The methods are inspired by \cite{MR3357585} and
\cite{MR3666815}.
The conclusions of Section~\ref{sec:upper_R}
coupled with those of 
Section~\ref{sec:lower} that 
concern one-dimensional processes
may be used to obtain estimates of
other important objects
related to the process,
like the density of its running supremum,
see the equivalences in \cite{LCJM-2019}.

We note that the lower estimates 
are delicate as  due to the non-symmetry it is
troublesome to detect the position of
the supremum of the density.
We systematically address that problem in Section~\ref{sec:lower}
in view of a recent progress made in 
\cite{TGKS-2019}.

We will now introduce our setting.
Let
$d\in\N$ and $Y=(Y_t)_{t\geq 0}$ be
a L{\'e}vy process  in $\Rd$
(\cite{MR1739520}). 
Recall that there is a well known
one-to-one
 correspondence
between L{\'e}vy processes in $\Rd$ and
the vaguely continuous convolution semigroups 
of probability measures $(P_t)_{t\geq 0}$ on $\Rd$.
The characteristic exponent $\LCh$ of $Y$ is defined by 
$$\mathbb{E}
e^{i\left<x,Y_t\right>}
=\int_{\Rd} e^{i\left<x,y\right>} P_t(dy)
=e^{-t\LCh(x)}\,,\qquad x\in\Rd\,,
$$ 
and equals
$$
\LCh(x)=
\left<x,Ax\right>
-i\left<x,\drf\right> - \int_{\Rd} \left(e^{i\left<x,z\right>}-1 - i\left<x,z\right>\ind_{|z|<1}\right)\LM(dz)\,.
$$
Here 
$A$ is a symmetric non-negative definite matrix, 
$\drf \in \Rd$ and
$\LM(dz)$ is a L{\'e}vy measure, i.e., 
a measure satisfying
$$
\LM(\{0\})=0\,,
\qquad\quad
\int_{\Rd} (1\land |z|^2)\LM(dz)<\infty \,.
$$
We call $(A,\LM,\drf)$ the generating triplet of $Y$. 
Our aim is to discuss  and establish estimates of the
density $p(t,x)$ of $Y_t$, or equivalently of $P_t(dx)$.
To this end 
for $r>0$ we define 
$$
h(r)=r^{-2} \|A\|+\int_{\Rd}\left(1\wedge\frac{|x|^2}{r^2}\right)\LM(dx)\,,
$$ 
and
$$
K(r)=r^{-2}\|A\|+r^{-2}\int_{|x|<r}|x|^2\LM(dx)\,.
$$
Note that $|e^{-t \LCh(x)}|= e^{-t {\rm Re}[\LCh(x)]}$ 
and if $e^{-t \LCh(x)}$ is absolutely integrable,
then we can invert the Fourier transform
and represent the density as
\begin{align*}
p(t,x)=(2\pi)^{-d}\int_{\R^d} e^{-i\left<x,z\right>}e^{-t\LCh (z)}\, dz\,.
\end{align*}

Let us provide the fundamental relation
between {\it the concentration function} $h$ 
and {\it the characteristic exponent} $\LCh$.
The real part of $\LCh$
equals
$
{\rm Re}[\LCh(x)]=\left<x,Ax\right>+\int_{\Rd}\big( 1-\cos \left<x,z\right> \big)\LM(dz)
$ and
we consider its radial, continuous and non-decreasing majorant
defined by
$$
\LCh^*(r)=\sup_{|z|\leq r} {\rm Re}[\LCh(z)],\qquad r>0\,.
$$
Then $h$ and $\LCh^*$ are comparable as follows
(see \cite[Lemma~4]{MR3225805}) 
\begin{align}\label{ineq:comp_TJ}
\frac{1}{8(1+2d)} h(1/r)\leq \LCh^*(r) \leq 2 h(1/r)\,,\qquad r>0\,.
\end{align}

It is a usual practice
to use  the characteristics
$(A,\LM,b)$, $h$ and $K$, $\LCh$ and $\LCh^*$,
describing a L{\'e}vy process,
in order to formulate assumptions and results
on the estimates of $p$
(see \cite{TGKS-2019}, \cite{MR3139314},~\cite{MR2995789}).
Note
that $h(0^+)<\infty$ ($h$ is bounded) if and only if $A=0$ and $\LM(\Rd)<\infty$, i.e., the corresponding L\'{e}vy process is a compound Poisson process (with drift). 
{\bf We assume in the whole paper that
$h(0^+)=\infty$}.

L{\'e}vy processes that we focus on
are such that the density $p(t,x)$ exists and
(with the
exception of
Theorem~\ref{thm:up_d1},
Lemma~\ref{lem:low_general_Rd}
and~\ref{lem:low_general_R1},
where weaker assumptions are imposed)
satisfies
$$
\sup_{x\in\Rd} p(t,x) \leq c_1 \left[h^{-1}(1/t)\right]^{-d},
$$
for a constant $c_1>0$ and all $t<T_1$, where $T_1\in (0,\infty]$ is fixed.
Such a property of a L{\'e}vy process was named  $\Ca$ and 
thoroughly investigated in
\cite[Section~3.1]{TGKS-2019}. 
In particular, several non-trivial equivalent reformulations 
called $\Cb$~--~$\Ch$,
which involve $p$, $h$ and $K$, $\LCh$ and $\LCh^*$, were given.
They all validate the Fourier inversion formula for the density.
In the literature one finds
even more  conditions of the same meaning, 
which arise
due to the relation between $h$ and $\LCh^*$ in
\eqref{ineq:comp_TJ}.
We include two of them, labelled $\Nb$ and $\Nc$, into our discussion, see Lemma~\ref{lem:eqiv_h_LCh}.
To be consistent with the existing literature,
the central assumption in all our results, with the aforementioned exceptions, becomes either 
the condition $\Cc$ or  $\Nb$ restated in
\eqref{KP_as8}.
It is noteworthy,
as we give certain importance to L{\'e}vy processes on the real line, that
in one dimension 
 $\Cc$ and
\eqref{KP_as8} lighten to a scaling
condition for the function $h$,
see \cite[Remark~3.2]{TGKS-2019}.

We emphasize  that a large portion of {\bf our results 
constitute  a toolbox}
to choose from
 according to the process under  consideration.
For upper bounds
those tools are
Theorem~\ref{thm:KP},
Proposition~\ref{prop:one_side}
and
Theorem~\ref{thm:up_d1},
which could also be supported by
\cite[Theorem~3.1]{TGKS-2019}.
For lower bounds these are
various results of
Section~\ref{sec:lower}
together with
\cite[Theorem~5.2 and~5.3]{TGKS-2019}.
For example,
if $d=1$ and a  L{\'e}vy process corresponds to a triplet $(A,\LM,\drf)$, where $A=0$, $\drf\in\RR$ and
$$
\LM(dx)=n(x)dx\,,\qquad
n(x)\approx
|x|^{-1-\alpha}
\left( e^{-c_-\,|x|}\ind_{x<0}
+
e^{- c_+\,|x|}\ind_{x>0}
\right),
$$
with
$\alpha\in(0,2)$
and $c_-,c_+\geq 0$,
we deduce 
from our results that for
every $T>0$, 
$$
p(t,x+t \drf_{[1/\LCh^{-1}(1/t)]})
\approx \left( t^{-1/\alpha} \land t n(x)\right),
$$
on $(0,T)\times \RR$,
see Example~\ref{ex:5}.
The necessary notation is 
gathered at the end of this section.
The non-symmetry of $\LM(dx)$
is responsible for the appearance of a shift in the spacial coordinate, which 
for $r>0$ is defined by
\begin{align}\label{def:br}
\drf_r=\drf+\int_{\Rd} z \left(\ind_{|z|<r} - \ind_{|z|<1}\right)\LM(dz)\,.
\end{align}
For more applications in the same spirit we refer the reader to Examples~\ref{ex:3} -- \ref{ex:4b}.

We
present yet
another sample of our  findings,
this time in $\Rd$,
which is an excerpt from Section~\ref{sec:cU}.
Related applications can be found in Examples~\ref{ex:1} -- \ref{ex:2}.

\begin{theorem}\label{thm:int1}
Let  $A=0$, $b\in\Rd$ and $\LM(dx)=\gLM(x)dx$ be a L{\'e}vy measure  
satisfying $\gLM(x)\approx g(|x|)$ for a non-increasing function $g\colon [0,\infty)\to [0,\infty]$.
Given $T \in (0,\infty]$ the following are equivalent:\\

\begin{enumerate}
\item[\it (a)]
There is $c\in [1,\infty)$ such that for all $r<T$,
$$
h(r) \leq c K(r)\,.
$$

\item[\it (b)]
There is $c>0$ such that for all $t<T$,
$$
\int_{\Rd} e^{-t\, {\rm Re}[\LCh(z)]} dz \leq c \left[h^{-1}(1/t)\right]^{-d}.
$$

\item[\it (c)]
For every $\bbbeta\in \mathbb{N}_0^d$
there is $c\in[ 1,\infty)$ such that
for all $0<t<T$, $x\in\Rd$,
\begin{align*}
| \partial_x^{\bbbeta}\, p(t,x+t\drf_{[h^{-1}(1/t)]}) | \leq c  \left[ h^{-1}(1/t) \right]^{-|\bbbeta|} \left( [h^{-1}(1/t)]^{-d}\land \frac{t K(|x|)}{|x|^{d}}\right),
\end{align*}
and
\begin{equation*}
 p(t,x+t\drf_{[h^{-1}(1/t)]})\geq c^{-1} \bigg( [ h^{-1}(1/t)]^{-d}  \land  \,t \gLM (x)\bigg).
\end{equation*}
\end{enumerate}
Furthermore,
there are  $c \in[1,\infty)$, $R\in (0,\infty]$ such that for all $|x|<R$,
$$\frac{K(|x|)}{|x|^d} \leq c\, \gLM(x)\,.$$
if and only if 
for some $\beta\in [0,2)$, $c\in(0,1]$
we have
$c\, \lambda^{d+\beta} g(\lambda r) \leq g(r)$,
$\lambda \leq 1$, $r<R$.
\end{theorem}

The assertions (a) and  (b) of Theorem~\ref{thm:int1}
may be restated in other equivalent forms by 
\cite[Lemma~2.3, Theorem~3.1, Proposition~3.6, Corollary~3.8 and~3.9, Lemma~3.9]{TGKS-2019}.
The last part of the theorem may be reformulated using Lemma~\ref{lem:comp_uK_uLM}.
The proof of Theorem~\ref{thm:int1} is given at the end of Section~\ref{subsec:unim_comp}.

The remainder of the paper is organized as follows.
In Section~\ref{sec:gen} we refine 
\cite[Theorem~3]{MR3357585} and give further general consequences.
In Section~\ref{sec:upper_R}
we prove upper estimates in dimension $d=1$.
Section~\ref{sec:lower} is dedicated to the lower  estimates.
In Section~\ref{sec:cU} we give
upper estimates 
of the density and its derivatives
for  L{\'e}vy processes
comparable, in a sense, with an isotropic unimodal
L{\'e}vy process,
and we complement them with lower estimates.
In Section~\ref{sec:exls} we present concrete examples.
In Section~\ref{sec:uni_app} we store supplementary results
on
 isotropic unimodal L{\'e}vy 
processes.
In particular, we assemble crucial features of {\it the bound function} $\rr_t(x)$ defined in \eqref{def:bound_function}.
We end the paper with auxiliary results in Section~\ref{sec:aux}.

Now we fix
the notation.
Throughout the article
$\omega_d=2\pi^{d/2}/\Gamma(d/2)$ is the surface measure of the unit sphere in $\R^d$.
$B(x,r)$ is a ball of radius $r$ centred at $x$, and $B_r=B(0,r)$.
By $c(d,\ldots)$ we denote a generic
 positive constant that depends only on the listed parameters $d,\ldots$. 
We write $f(x)\approx g(x)$, or simply $f\approx g$,
if there is a constant $c\in [1,\infty)$ independent of $x$ such that
$c^{-1} f(x)\leq g(x)\leq c f(x)$.
As usual $a\land b=\min\{a,b\}$ and $a\vee b = \max\{a,b\}$.
For a set $\setA \subseteq \Rd$ we denote 
$\delta(\setA)=\inf \{|y|\colon y\in \setA\}$ and
${\rm diam}(\setA)=\sup\{|y-x|\colon x,y\in \setA \}$.
Borel sets in $\Rd$ will be denoted by $\mathcal{B}(\Rd)$.
A Borel measure $\nu$ on $\Rd$ is called symmetric if
$\nu(\setA)=\nu(-\setA)$ for every $\setA\in \mathcal{B}(\Rd)$.
For $s>0$ let $\LCh^{-1}(s)=\sup\{r>0\colon \LCh^*(r)=s\}$
so that $\LCh^*(\LCh^{-1}(s))=s$ and $\LCh^{-1}(\LCh^*(s))\geq s$.

\section{Upper bounds in $\Rd$}
\label{sec:gen}

In this section we consider a L{\'e}vy process $Y$ in $\Rd$
with a generating triplet $(A,\LM,\drf)$, where $A=0$.
In Theorem~\ref{thm:KP} below
we
refine
\cite[Theorem~3]{MR3357585}
by
taking into
account the results of 
\cite{TGKS-2019}
and
tracing
the dependence of constants, 
which has not been explicitly done 
in~\cite{MR3357585}.

\begin{theorem}\label{thm:KP}
Assume that $\LM(dx)$ is a L{\'e}vy measure such that 
\begin{align}\label{KP_as1}
\LM(\setA)\leq  f(\delta(\setA)) [{\rm diam}(\setA)]^{\gamma},\qquad \setA\in \mathcal{B}(\Rd),
\end{align}
where $\gamma\in [0,d]$, and $f\colon [0,\infty)\to [0,\infty]$ is non-increasing function satisfying
\begin{align}\label{KP_as2}
\int_{|y|>r} f\left(s\vee |y|-\frac{|y|}{2} \right)\LM(dy)\leq M_2 f(s) \LCh^*(1/r),\qquad s>0,\,r>0,
\end{align}
for some constant $M_2>0$. Further, assume that there are $M_5>0$
and $T\in (0,\infty]$ 
such that
\begin{align}\label{KP_as8}
\int_{\Rd} e^{-t \, {\rm Re}[\LCh(\xi)]} \,d\xi \leq M_5\, [\LCh^{-1}(1/t)]^{d},\qquad t\in(0,T).
\end{align} 
Then for all $k,m\in \mathbb{N}_0$ 
satisfying
$m\geq k>\gamma$ and every $\bbbeta \in \mathbb{N}_0^d$ such that $|\bbbeta|<m-k$
there is a constant $C=C(d,M_2,M_5,\gamma,m,k)$ such that for all $x\in\Rd$, $t\in (0,T/[16(1+2d)])$,
\begin{align*}
|\partial_x^{\bbbeta} p(t&,x +t \drf_{[1/\LCh^{-1}(1/t)]} )|\\[5pt]
&\leq C \left[\LCh^{-1}(1/t)\right]^{d+|\bbbeta|} \min\left\{1, t \left[\LCh^{-1}(1/t)\right]^{-\gamma} f(|x|/4)+\left(1+|x|\LCh^{-1}(1/t) \right)^{-k}\right\}\,.
\end{align*}
\end{theorem}
\pf
Using Lemma~\ref{lem:eqiv_h_LCh} we get from \eqref{KP_as8} that
$\int_{\Rd}|\xi|^m e^{-t \, {\rm Re}[\LCh(\xi)]}  \,d\xi \leq \tilde{M}_5\, [\LCh^{-1}(1/t)]^{d+m}$
for all $t< T/[16(1+2d)]$ 
with $\tilde{M}_5=\tilde{M}_5(d,m,M_5)$.
Thus, the remaining inequality (8) from 
\cite[Theorem~3]{MR3357585}
holds.
Now we carefully inspect  the dependence of constants in the proof of 
\cite[Theorem~3]{MR3357585}
following the notation therein.
The constant $c_1$ depends only on dimension and $M_5$. Next, the constant $c_3=c_2/(1-\gamma/n)$ and $c_2$ depends on $C_{19}$, $M_2$, $\gamma$ and the dimension. Finally, $C_{19}$ depends only  on $M_5,m,n$ and the dimension.
\qed

In 
Corollary~\ref{cor:KP_imply}
we show
that Theorem~\ref{thm:KP} leads to a very general result 
that merely 
requires \eqref{KP_as8} to hold.
It was shown in \cite[Theorem~3.1]{TGKS-2019} that the latter (see also Lemma~\ref{lem:eqiv_h_LCh}) is equivalent
to the assumption {\bf A} of \cite{MR3139314} and \cite{MR3235175}.
In this context, 
it is noteworthy that
for transition densities discussed in
\cite[Example~4.2]{MR3139314} and \cite[Example~2]{MR3235175}
the result of
Corollary~\ref{cor:KP_imply}
provides the same 
"bell-like" estimates 
like those
obtained
in
\cite{MR3139314} and \cite{MR3235175}
by
applying
theorems 
that rely on additional assumption
of
sub-exponential distribution function.
In those extreme examples, the estimates are
known to be precise 
"bell-like" estimates.
On the other hand,  
for the isotropic $\alpha$-stable process
the upper bound in Corollary~\ref{cor:KP_imply}
will not coincide with the well known sharp estimates
of its transition density.

%

\begin{corollary}\label{cor:KP_imply}
Assume that \eqref{KP_as8} holds for some $M_5>0$ and $T\in (0,\infty]$.
Then for every $\bbbeta\in \mathbb{N}_0^d$ there is a constant $C=C(d,M_5,\bbbeta)$
such that for all $x\in\Rd$, $t\in (0,T/[16(1+2d)])$,
\begin{align*}
|\partial_x^{\bbbeta} p(t&,x +t \drf_{[1/\LCh^{-1}(1/t)]} )|\\
&\leq 
C \left[\LCh^{-1}(1/t)\right]^{d+|\bbbeta|} 
\min\left\{ 1 ,\,  t h(|x|)\right\}\,.
\end{align*}
\end{corollary}
\pf
Let $f(s)=h(s)$, $s>0$. 
We use \cite[Lemma~2.1]{TGKS-2019} repeatedly.
Plainly, 
\eqref{KP_as1} holds with $\gamma=0$.
Since $f$ is decreasing $f(s\vee |y|-|y|/2)\leq f((s\vee|y|)/2)\leq f(s/2)
\leq 4 f(s)$ for $s>0$, $y\in\Rd$.
Thus \eqref{KP_as2} holds with $M_2= 2 c_d$ by \eqref{ineq:comp_TJ}.
 Here $c_d=16(1+2d)$.
Now we apply Theorem~\ref{thm:KP} with $m=|\bbbeta|+2$ and $k=2$.
If $|x|\leq h^{-1}(1/t)$, the desired inequality holds.
If
$|x|> h^{-1}(1/t)$, 
we first note that \eqref{ineq:comp_TJ_inverse} and \cite[Lemma~2.4]{TGKS-2019}
give \cite[(A1)]{TGKS-2019} and thus by \cite[Lemma~2.3]{TGKS-2019} also \cite[(A2)]{TGKS-2019} with some
$\lah=\lah(d,M_5)$, $C_h=C_h(d,M_5)$ and $\theta_h=h^{-1}(c_d/(2 T))$.
Therefore there is $c=c(d,M_5)$ such that $h^{-1}(1/(2t))\leq c h^{-1}(1/t)$
for all $t\in (0,T/c_d)$, which together with \eqref{ineq:comp_TJ_inverse} imply
\begin{align*}
\left(1+|x|\LCh^{-1}(1/t) \right)^{-2}\leq 
\left(\frac{h^{-1}(1/(2t))}{|x|}\right)^2
\leq c^2 \left(\frac{h^{-1}(1/t)}{|x|}\right)^2\,.
\end{align*}
Next, recall that
$\lambda^2\leq h(r)/h(\lambda r)$, $\lambda \leq 1$, $r>0$.
Taking $\lambda= h^{-1}(1/t)/|x|$ and $r=|x|$ we get
$$
\left(\frac{h^{-1}(1/t)}{|x|}\right)^2 \leq t h(|x|)\,,
$$
for all $t\in (0,T/c_d)$. Since $f(|x|/4)\leq 4^2 h(|x|)$, we can bound the minimum by $th(|x|)$. This ends the proof.
\qed

\section{Upper bounds in $\RR$}
\label{sec:upper_R}

Let $Y$ be a L{\'e}vy process in $\RR$ corresponding to a L{\'e}vy triplet
$(A,\LM,\drf)$.
The aim of this section is to discuss processes with L{\'e}vy measures that exhibit entirely different behaviour
on the positive and on the negative half-line
so that they cannot be treated well by Theorem~\ref{thm:KP}.

The first result is motivated by \cite[Theorem~1]{MR3357585}.

\begin{proposition}\label{prop:one_side}
Assume that $\LM(dx)$ is a L{\'e}vy measure such that 
\begin{align*}
\LM(\setA)&\leq  f_+(\delta(\setA)) [{\rm diam}(\setA)]^{\gamma_+},
\qquad \setA\in \mathcal{B}((0,\infty)),
\end{align*}
where $\gamma_+\in [0,1]$
and $f_+\colon (0,\infty)\to (0,\infty)$ is a non-increasing function
satisfying
\begin{align*}
f_+(x/2)\leq c_{f_+} f_+(x)\,,\qquad x > 0\,,
\end{align*}
for some constant $c_{f_+} \geq 1$. 
Further, assume that 
\eqref{KP_as8}
holds for some $M_5>0$
and $T\in (0,\infty]$.
Then
there are absolute non-negative constants
$c_1$, $c_2$
and a constant  $C=C(M_5,c_{f_+}, \gamma_+)$,
such that for all $t\in (0,T/48)$
and $x > 0$,
\begin{align*}
p(t&,x +t \drf_{[1/\LCh^{-1}(1/t)]} )\\
&\leq 
C \left[ \LCh^{-1}(1/t)\right]
\min\left\{1, t \left[\LCh^{-1}(1/t)\right]^{-\gamma_+} f_+(x)+ e^{-c_1 \LCh^{-1}(1/t)x \log(1+c_2\LCh^{-1}(1/t)x)} \right\}\,.
\end{align*}
\end{proposition}
\pf
Let $\LM_r(dx)=\LM|_{B_r^c}(dx)$. We first prove that
for all $\setA \in \mathcal{B}((0,\infty))$
 and $n\in\N$,
\begin{align}\label{ineq:tranc_conv}
\LM_r^{n*} (\setA) \leq n\,c_{f_+}^{n-1}\,   
[h(r)]^{n-1}  f_+(\delta(\setA)) [{\rm diam}(\setA)]^{\gamma_+} \,.
\end{align}
The inequality holds for $n=1$.
We proceed by induction.
Let $D=(\delta(\setA)/2,\infty)$, then
\begin{align*}
\LM_r^{(n+1)*}(\setA)
&=
\int_{\RR} \LM_r((\setA-y)\cap D)\, \LM_r^{n*}(dy)
 + \int_{\RR} 
\LM_r((\setA-y)\cap D^c)\, \LM_r^{n*}(dy) =: I_1+I_2\,.
\end{align*}
By the monotonicity of $f_+$ and our assumption we get
$I_1\leq  f_+(\delta(\setA)/2) [{\rm diam}(\setA)]^{\gamma_+} \LM_r^{n*}(\RR)$.
Next, since the convolution is commutative,
\begin{align*}
I_2= \int_{\RR} \LM_r^{n*}(\setA-z) \,\LM_r(dz \cap D^c)
\leq 
n\, c_{f_+}^{n-1}\,  [h(r)]^{n-1} f_+(\delta(\setA)/2) [{\rm diam}(\setA)]^{\gamma_+}\, \LM_r(\RR)\,.
\end{align*}
Finally, note that $\LM_r^{n*}(\RR)\leq [h(r)]^n$. 
This gives \eqref{ineq:tranc_conv}, which with \eqref{ineq:comp_TJ} further imply
\begin{align}\label{ineq:tranc_conv_ball}
\LM_r^{n*}(B(x,\rho))\leq 
n\, 24^{n-1}\, c_{f_+}^{n}     [\LCh^*(1/r)]^{n-1}  f_+(x)  (2\rho)^{\gamma_+}
\end{align}
for all $x>0$ and $0<\rho<x/2$.
Now, Lemma~\ref{lem:eqiv_h_LCh}
shows that \eqref{KP_as8} implies
\cite[(3)]{MR3357585}
with $M_3=M_3(d,M_5)$ for all
$t<T/48$.
Therefore the statement of 
\cite[Lemma~8]{MR3357585}
holds true for all $x\in\Rd$ and $t<T/48$,
and consequently our result has the same proof as
\cite[Theorem~1]{MR3357585},
only \cite[Corollary~10]{MR3357585}
has to be replaced by \eqref{ineq:tranc_conv_ball}.
\qed

The statement of Proposition~\ref{prop:one_side} 
provides estimates of the transition density for arguments
to the right of
$t \drf_{[1/\LCh^{-1}(1/t)]}$.
Similarly, given the knowledge of the L{\'e}vy measure on the negative half-line 
leads to estimates for arguments to the left of $t \drf_{[1/\LCh^{-1}(1/t)]}$.
Such result follows from
Proposition~\ref{prop:one_side} applied 
to a L{\'e}vy process generated by $(A,\tilde{\LM},\tilde{\drf})$, where $\tilde{\drf}=-\drf$
and $\tilde{\LM}(dx)=\LM(-dx)$ is a reflected measure.

\begin{corollary}\label{cor:one_side}
Assume that $\LM(dx)$ is a L{\'e}vy measure such that 
\begin{align*}
\LM(-\setA)&\leq  f_-(\delta(\setA)) [{\rm diam}(\setA)]^{\gamma_-},
\qquad \setA \in \mathcal{B}((0,\infty)),
\end{align*}
where  $\gamma_-\in [0,1]$
and $f_-\colon (0,\infty)\to (0,\infty)$ is a non-increasing function
satisfying
\begin{align*}
f_-(x/2)\leq c_{f_-} f_-(x)\,,\qquad x>0\,,
\end{align*}
for some constant $c_{f_-} \geq 1$.
Further, assume that 
\eqref{KP_as8}
holds for some $M_5>0$
and $T\in (0,\infty]$.
Then for all $t\in (0,T/48)$
and $x > 0$,
\begin{align*}
 p(t&,-x +t \drf_{[1/\LCh^{-1}(1/t)]} )\\
&\leq 
C \left[ \LCh^{-1}(1/t)\right]
\min\left\{1, t \left[\LCh^{-1}(1/t)\right]^{-\gamma_-} f_-(x)+ e^{-c_1 \LCh^{-1}(1/t)x \log(1+c_2\LCh^{-1}(1/t)x)} \right\}\,,
\end{align*}
where $c_1$, $c_2$ and $C=C(M_5,c_{f_-}, \gamma_-)$ are taken from Proposition~\ref{prop:one_side}.
\end{corollary}

A function $f$ in Proposition~\ref{prop:one_side} or Corollary~\ref{cor:one_side} cannot have exponential decay,
which is due to the condition $f(x/2)\le c f(x)$.
Here we propose a highly non-symmetric one-dimensional version of \cite[Theorem~1]{MR3666815}
to capture the decay of the transition density $p(t,x)$
for large $x$ for such measures.
Note also that
the condition
\eqref{KP_as8}, 
assumed in the results above,
implies ${\rm Re}[\LCh(x)]\geq \underline{c}\, |x|^{\underline{\alpha}}$ for some $\underline{c},\underline{\alpha} >0$ and large $|x|$,
see Lemma~\ref{lem:eqiv_h_LCh}.
We use the latter instead of \eqref{KP_as8} in our next result.

\begin{theorem}\label{thm:up_d1}
Assume that there are $\underline{c},\underline{\alpha}>0$ and $x_0\geq 1$ such that
\begin{align*}
{\rm Re}[\LCh(x)]\geq \underline{c}\, |x|^{\underline{\alpha}}\,,\qquad |x|> x_0\,.
\end{align*}
Suppose $\LM(dx)$ is  a L{\'e}vy measure
 that for some $r_0, c_0\geq 1$ satisfies 
$$\LM(dx)\leq f_+(x)dx\,,\qquad x>r_0\,,$$
\begin{align*}
\int_{r_0}^{x-r_0} f_+(x-z) f_+(z)\,dz\leq  c_0\, f_+(x)\,,\qquad  x>2 r_0\,,
\end{align*}
for a non-increasing function
$f_+\colon (0,\infty)\to (0,\infty)$.
Given $t_0>0$ there exists a constant $c>0$ such that
\begin{align*}
p(t,x+t \drf) \leq c t f_+(x)\,,
\end{align*}
for all $t\in (0,t_0]$ and
$x\geq 4 t_0 (v_1+ v_2 +2A)+4\delta$, where
$$
v_1=\int_{[1,r_0]} z \LM(dz)\,,\qquad v_2=\int_{(-1,r_0]}z^2\LM(dz)\,,\qquad 
\delta=4(r_0+1)(1+1/\underline{\alpha})\,.
$$
\end{theorem}

\begin{remark}\label{rem:x_neg1}
Obviously, Theorem~\ref{thm:up_d1}
has a version for $x<0$, just like
Corollary~\ref{cor:one_side}
is such for Proposition~\ref{prop:one_side},
and it is obtained
from Theorem~\ref{thm:up_d1} by considering
$(A,\tilde{\LM},\tilde{\drf})$, where $\tilde{\drf}=-\drf$
and $\tilde{\LM}(dx)=\LM(-dx)$.
\end{remark}

Before we give the actual proof of Theorem~\ref{thm:up_d1}
we establish a few auxiliary results.

\begin{lemma}\label{lem:f_aux}
Assume that $f\colon (0,\infty)\to (0,\infty)$
is a non-increasing function.

\noindent
a) If for some $c_0>0$ and $r_0\geq 1$,
\begin{align*}
\int_{r_0}^{x-r_0} f(x-z) f(z)\,dz\leq  c_0 f(x)\,,\qquad  x>2 r_0\,,
\end{align*}
then there is $c_a\geq 1$ such that 
\begin{align}\label{ineq:f_shift}
f(s-r_0)\leq c_a f(s)\,, \quad s\geq 2 r_0\,.
\end{align}

\noindent
b) If \eqref{ineq:f_shift} holds for some $c_a,r_0 \geq 1$, then
given $c_1,c_2>0$ there exists $c_b$ such that
$$\exp\left(- c_1 s \ln\left(c_2 s\right)  \right)\leq c_b f(sr_0)\,, \quad c_2 s\geq 1\,.$$
\end{lemma}
\pf
For $s\in [2r_0, 3r_0)$ we have $f(s)/f(s-r_0)\geq f(3r_0)/f(r_0)$. For $s\geq 3r_0$ we have
$$
c_0 f(s)\geq \int_{s-2r_0}^{s-r_0} f(s-z)f(z)\,dz
\geq f(s-r_0)f(2r_0)r_0\,.
$$
The first part of the statement follows with $c_a=\max\{1,\frac{c_0}{f(2r_0)r_0},\frac{f(r_0)}{f(3r_0)}\}$.
Now we prove the second part. Set
$$
n_0=\inf\{n\in\N\colon   \left(c_a^{-1/c_1} c_2 n \right)^{-c_1 n} \leq f(r_0)  \,,\quad  c_2 n  \geq c_a^{1/c_1}  \}\,.
$$
Then for $s\geq n_0$ we let $n\geq n_0$ such that $s\in [n,n+1)$, and we get
\begin{align*}
\exp\left(-c_1 s \ln(c_2 s) \right)
\leq \exp\left(-c_1 n \ln(c_2 n) \right)
\leq c_a^{-n} f(r_0) \leq f(r_0+nr_0)\leq f(sr_0)\,.
\end{align*}
Including $s\in [c_2^{-1},n_0]$ 
we have that the inequality holds with
$c_b=\max\{1, \frac{1}{f(n_0r_0)}\}$.
\qed

In what follows we want to derive some consequences of  \cite[Lemma~1]{MR2794975}. 
Due to a slight inaccuracy in the proof of this result, we will first provide full proof of it in dimension $d=1$.
In Lemma~\ref{lem:comp_sup}
we denote by $Z^0=(Z_t^0)_{t\geq 0}$ 
a L{\'e}vy process that
corresponds to $(A,\LM_0,0)$, where
$\LM_0$ is a L{\'e}vy measure such that
$\supp \LM_0 \subseteq [-r_0,r_0]$, $r_0>0$.
\begin{lemma}\label{lem:comp_sup}
For every $a \geq 2(\max\{0,\xi_0\}+(2A+M_0)/e)$ we have
\begin{align*}
\PP(Z_1^0 \geq a)\leq  \exp\left( -\frac{a}{2(r_0+1)} \ln\frac{a}{2(2A+M_0)}\right),
\end{align*}
where $\xi_0=\int_{|z|\geq 1} z \LM_0(dz)$ and $M_0=\int_{\RR} z^2\LM_0(dz)$.
\end{lemma}
\pf
For $\varepsilon>0$ let $\LM_{\varepsilon}(dz)=\LM_0(dz)+\varepsilon \delta_{\{r_0/(r_0+1) \}}(dz)$.
We consider $Z^{\varepsilon}=(Z_t^{\varepsilon})_{t\geq 0}$ that corresponds to $(A,\LM_{\varepsilon},0)$.
Following \cite[Lemma~26.4 and~26.5]{MR1739520} let
$$\psi(u)= 2A u + \int_{\RR} z \left(e^{uz} - \ind_{|z|< 1}\right) \LM_{\varepsilon}(dz)\,,\qquad u> 0\,.$$
Note that $\psi$ is increasing,
$\xi_0=\lim_{u\to 0^+} \psi(u)$
and
$\lim_{u\to\infty}\psi(u)=\infty$.
Therefore, by
\cite[Lemma~26.4]{MR1739520}
for $a>\xi_0$ we have
\begin{align*}
\PP(Z_1^{\varepsilon} \geq a)\leq \exp\left(-\int_{\xi_0}^a \theta (\xi)\,d\xi\right),
\end{align*}
where 
$\theta(\xi)$ on $(\xi_0,\infty)$ is an inverse function of $\psi(u)$ on $(0,\infty)$.
For $\xi> \xi_0$,
\begin{align*}
\xi= \psi(\theta(\xi))
&= 2A\theta(\xi)
+\int_{\RR}
z \left(e^{\theta(\xi) z} - \ind_{|z|< 1}\right) \LM_{\varepsilon}(dz)\\
&=
2A\theta(\xi)
+\int_{\RR}
z \left(e^{\theta(\xi) z} - 1\right) \LM_{\varepsilon}(dz)
+\xi_0\\
&\leq 2A\theta(\xi)+e^{\theta(\xi)r_0}\theta(\xi)\int_{\RR} z^2\LM_{\varepsilon}(dz)+\xi_0
\leq  \frac{1}{e}(2A+M_{\varepsilon})e^{\theta(\xi)(r_0+1)}+\xi_0\,,
\end{align*}
where $M_{\varepsilon}=\int_{\RR} z^2\LM_{\varepsilon}(dz)$.
Thus
$$
\theta(\xi)\geq  \frac{1}{r_0+1} \ln \left( e\, \frac{\xi-\xi_0}{2A+M_{\varepsilon}} \right).
$$
Consequently, for $a>  \xi_0$,
\begin{align*}
\PP(Z_1^{\varepsilon}\geq a)\leq \exp\left[-\frac{1}{r_0+1}\int_{\xi_0}^{a}\ln\left( e\, \frac{\xi-\xi_0}{2A+M_{\varepsilon}}\right)d\xi \right]
=\exp\left[-\frac{a-\xi_0}{r_0+1} \ln\left( \frac{a-\xi_0}{2A+M_{\varepsilon}} \right) \right].
\end{align*}
Note that $Z_1^{\varepsilon}$ has the same distribution as $Z_1^0+S$, where $S\geq 0$. 
Hence
\begin{align*}
\PP(Z_1^0 \geq a)\leq \PP(Z_1^{\varepsilon}\geq a)\leq
\exp\left[-\frac{a-\xi_0}{r_0+1} \ln\left( \frac{a-\xi_0}{2A+M_{\varepsilon}} \right) \right].
\end{align*}
Passing with $\varepsilon$ to zero this yields
\begin{align*}
\PP(Z_1^0 \geq a)\leq
\exp\left[-\frac{a-\xi_0}{r_0+1} \ln\left( \frac{a-\xi_0}{2A+M_0} \right) \right],\qquad  a>\xi_0\,.
\end{align*}
Now, let $a\geq 2(\max\{0,\xi_0\}+(2A+M_0)/e)$.
Since
$a>\xi_0$ and
the function $s\ln(s)$ is increasing on $[1/e,\infty)$, we obtain the desired inequality by $a-\xi_0 \geq a/2 \geq (2A+M_0)/e$.
\qed

In Corollary~\ref{lem:tail}, \ref{cor:tail} and~\ref{cor:upper}
we denote by $Z^1=(Z_t^1)_{t\geq 0}$ 
a L{\'e}vy process that
corresponds to $(A,\LM_1,0)$, where
$\LM_1$ is a L{\'e}vy measure such that
$\supp \LM_1 \subseteq (-\infty,r_0]$, $r_0>0$.
The result is motivated by \cite[Theorem~26.8]{MR1739520}.
\begin{corollary}\label{lem:tail}
For every $a\geq 2 (\xi_1+\xi_2+2A)$ we have
\begin{align*}
\PP(Z_1^1 \geq a)\leq  \exp\left( -\frac{a}{2(r_0+1)} \ln\frac{a}{2(\xi_2+2A)}\right),
\end{align*}
where $\xi_1=\int_{[1,\infty)} z \LM_1(dz)$ and $\xi_2=\int_{(-1,\infty)}z^2\LM_1(dz)$.
\end{corollary}
\pf
It suffices to notice that $Z_1^1$ has the same distribution as $Z_1^0+S$, where $S\leq 0$
and $(Z_t^0)_{t\geq 0}$ corresponds to $(A,\LM_0,0)$ with $\LM_0(dz)=\LM_1|_{\{x> - (1\land r_0)\}}(dz)$.
Then use $\PP(Z_1^1 \geq a)\leq \PP(Z_1^0 \geq a)$ and apply Lemma~\ref{lem:comp_sup}. Clearly, $\xi_2+2A>0$ since we assume $h(0^+)=\infty$.
\qed

\begin{corollary}\label{cor:tail}
Let $\delta \geq 0$.
For every $a\geq 2 (\xi_1+\xi_2 +2A+\delta)$ we have
\begin{align*}
\PP(Z_1^1 \geq a)\leq 
\left(\frac{2(\xi_2+2A)}{a} \right)^{\delta/(2r_0+2)}
 \exp\left( -\frac{a}{4(r_0+1)} \ln\frac{a}{2(\xi_2+2A)}\right)  .
\end{align*}
\end{corollary}

The proof of the next result is the same as that of \cite[Lemma~2]{MR2794975}.

\begin{corollary}\label{cor:upper}
Assume that $Z_1^1$ has a density $p(x)$ such that
\begin{align*}
\sup_{x\in\RR} p(x)\leq m_0\,, \qquad \sup_{x\in\RR} |p'(x)| \leq m_1.
\end{align*}
For $x \geq \max\{4(\xi_1+\xi_2 +2A+\delta) , \frac{m_0}{2m_1}\}$ we have
\begin{align*}
p(x)\leq (2m_1)^{1/2}
\left(\frac{4(\xi_2+2A)}{x} \right)^{\delta/(4r_0+4)}
 \exp\left( -\frac{x}{16(r_0+1)} \ln\frac{x}{4(\xi_2 +2A)}\right).
\end{align*}
\end{corollary}

\noindent
{\bf Proof of Theorem~\ref{thm:up_d1}}
First we decompose the L{\'e}vy process $Y$.
Consider 
\begin{align*}
\LM_{1}(dx):=\LM(dx)-\LM_{2}(dx),\qquad  
\LM_{2}(dx):=  \LM|_{\{x> r_0\}}(dx)\,.
\end{align*}
We let $Z^{1}$ and $Z^{2}$  be L{\'e}vy processes corresponding to
$(A,N_{1},\drf)$
and $(0,N_{2},0)$, respectively.
Similarly, we write $\LCh_{1}$
and $\LCh_{2}$
for their characteristic exponents.
The process $Z^{2}$ is a compound Poisson process
and we denote the distribution of $Z^2_t$ by $P_t^{2}$.
Note that 
$$P^2_t(dx)=e^{-t\LM_2(\RR)}\sum_{k=0}^{\infty} \frac{t^k}{k!} \LM_2^{*k}(dx)
=e^{-t\LM_2(\RR)} \delta_{\{0\}}(dx)+ e^{-t\LM_2(\RR)} p_2(t,x)dx\,,$$
and
\begin{align*}
p_2(t,x) dx \leq e^{c_0 t_0} \,t  f_+(x)\ind_{x>r_0} dx\,.
\end{align*}
Indeed, the latter follows easily by induction: 
since $\LM_2^{*k}(dx)$ is absolutely continuous and supported on $(k\, r_0,\infty)$ we write $\LM_2^{*k}(dx)=\LM_2^{*k}(x)dx$, and  we have
$\LM_2^{*(k+1)}(x)=0$ for $x\in (r_0,2r_0]$. For $x> 2r_0$ we get
\begin{align*}
\LM_2^{*(k+1)}(x)=\int_{r_0}^{x-r_0} \LM_2(x-z)\LM_2^{*k}(z)\,dz
\leq c_0^{k-1} \int_{r_0}^{x-r_0} f_+(x-z) f_+(z)\,dz\leq c_0^k f_+(x)\,.
\end{align*}
Further, observe that
\begin{align*}
{\rm Re}[\LCh_{1}(x)] \geq (\underline{c}/2) |x|^{\underline{\alpha}}\,,\qquad |x|>x_1:=\max\left\{x_0,
\left((4/\underline{c})\LM(r_0,\infty)\right)^{1/\underline{\alpha}}\right\}\,.
\end{align*}
Let
\begin{align*}
m_0=2 \int_{0}^{\infty}  e^{-t\, {\rm Re}[\LCh_1(z)]}\, dz\,,
\qquad
m_1= 2 \int_{0}^{\infty} (1\vee |z|) e^{-t\, {\rm Re}[\LCh_1(z)]}\, dz\,,
\end{align*}
and note that
$m_0/ m_1 \leq 1$.
In what follows a constant $c$ may change from line to line.
We  have
\begin{align*}
m_1
\leq 2x_1^2+2 \int_{x_1}^{\infty} |z| e^{- t (\underline{c}/2) |z|^{\underline{\alpha}}}\,dz 
\leq t^{-2/\underline{\alpha}} \left(2 x_1^2 t_0^{2/\underline{\alpha}} +2\int_0^{\infty} z e^{-(\underline{c}/2)|z|^{\underline{\alpha}}} \,dz \right)
= c t^{-2/\underline{\alpha}}.
\end{align*}
By
Corollary~\ref{cor:upper}
and Lemma~\ref{lem:f_aux}
the density $p_1(t,x)$ of $Z^1_t$ satisfies
\begin{align*}
p_1(t,x+t\drf)&
\leq  c t^{-1/\underline{\alpha}}
\left(\frac{4t(v_2+2A)}{x} \right)^{\delta/(4r_0+4)}
 \exp\left( -\frac{x}{16(r_0+1)} \ln\frac{x}{4t(v_2 +2A)}\right)
\\
&\leq 
 c
t^{-1/\underline{\alpha}+\delta/(4r_0+4)}\exp
\left( -\frac{x}{16(r_0+1)} \ln\frac{x}{4t_0(v_2+2A)}\right)
\leq c t f_+(x)\,,
\end{align*}
for all $t\in (0,t_0]$
and $x\geq r_1:=4t_0(v_1+ v_2 +2A)+4\delta$.
Therefore we have
\begin{align*}
p(t,x+t\drf)&=\int_{\RR}p_1(t,x-z+t\drf)P_t^2(dz)\\
&\leq p_1(t,x+t\drf)+  \int_{\RR} p_1(t,x-z+t\drf) p_2(t,z)\,dz\\
&\leq c t f_+(x)+ c t \int_{r_0}^{\infty}p_1(t,x-z+t\drf) f_+(z)\,dz\,.
\end{align*}
Now, if $x\in [r_1,r_1+r_0]$, then
\begin{align*}
\int_{r_0}^{\infty}p_1(t,x-z+t\drf) f_+(z)\,dz
\leq \frac{ f_+(r_0)}{f_+(r_1+r_0)}\, f_+(x)\,.
\end{align*}
If $x>r_1+r_0$, we have (note that $r_1>r_0$)
\begin{align*}
\int_{r_0}^{\infty}p_1(t,x-z+t\drf) f_+(z)\,dz
&=\int_{r_0}^{x-r_1} p_1(t,x-z+t\drf) f_+(z)\,dz
+\int_{x-r_1}^{\infty} p_1(t,x-z+tb) f_+(z)  \,dz\\
&\leq c t_0  \int_{r_0}^{x-r_1} f_+(x-z) f_+(z)\,dz
+  f_+(x-r_1)\,.
\end{align*}
Applying Lemma~\ref{lem:f_aux} $k$ times, where $k$ satisfies
$ k r_0 \geq r_1$, we obtain the final result.
\qed

\section{Lower bounds}
\label{sec:lower}

We consider a L{\'e}vy process $Y$ in $\Rd$ with a generating triplet $(A,\LM,\drf)$.
We begin with a general observation.
For necessary and sufficient conditions for the existence of the transition density see
\cite{MR182061}. We also refer the reader to \cite{MR3010850}.

\begin{lemma}\label{lem:low_general_Rd}
Assume that a L{\'e}vy measure $\LM(dx)$  is such that
for some $r_0\geq 1$, 
$$\LM(dx) \geq f(|x|)dx \,,\qquad |x|>r_0\,,$$
where 
$f\colon (0,\infty)\to [0,\infty)$
is a non-increasing function.  
There is a constant $c=c(d)$ such that
given $t_0>0$ there exists a constant $\tilde{c}>0$
so that 
$$
p(t,x+t\drf)\geq  \tilde{c}\, tf(|x|+r_1)\,,
$$
holds for each $t\in (0,t_0]$ 
such that
the transition density of $Y_t$ exists,
and  almost every $|x|\geq r_0+r_1$, where $r_1=h^{-1}(1/(2ct_0))$.
\end{lemma}

\pf
We decompose the L{\'e}vy process $Y$.
Consider 
\begin{align*}
\LM_{1}(dx):=\LM(dx)-\LM_{2}(dx),\qquad  
\LM_{2}(dx):= \LM|_{B_1^c}(dx)\,.
\end{align*}
We let $Z^{1}$ and $Z^{2}$  be L{\'e}vy processes corresponding to
$(A,\LM_{1},\drf)$
and $(0,\LM_{2},0)$, respectively.
The process $Z^{2}$ is a compound Poisson process
and we denote the distribution of $Z^2_t$ by $P_t^{2}$.
Note that distribution of $Z^1_t$ is absolutely continuous with a density $p_1(t,x)$ and
by
\cite{MR632968}
there is $c=c(d)$ such that for $r\geq 1$,
$$
\PP(|Z^1_t-t \drf |\geq r)\leq
 c t  \left( r^{-1} \left| \int_{\RR} z \left(\ind_{|z|<r} - \ind_{|z|<1}\right)\LM_1(dz)\right|  +h(r)\right)
= ct  h(r)\,.
$$
Therefore, by \cite[Lemma~2.1]{TGKS-2019} there is $r_1 \geq 1$
for which $\PP(|Z^1_t-t \drf |\geq r_1)\leq c t_0 h(r_1)=1/2$.
Then for $|x|\geq r_1+r_0$,
\begin{align*}
p(t,x+t \drf)
&= \int_{\RR} p_{1}(t,x+t \drf-z) P_t^{2}(dz)\\
&\geq  \int_{|x-z|<r_1} p_{1}(t,x+t \drf-z)\, e^{-t \LM_{2}(\RR)}\, t \LM_2(dz)\\
& \geq \int_{|z|< r_1}  p_{1}(t,z+t \drf)\, e^{-t \LM_{2}(\RR)}\, t f(|x-z|) \,dz\\
& \geq \left(\int_{|z|<r_1}  p_{1}(t,z+t \drf)dz\right)
  e^{-t_0 \LM(B_1^c) } \, t f(|x|+r_1)\\
&\geq   \PP(|Z^1_t-t \drf |< r_1)\,  e^{-t_0\LM(B_1^c) } \, t f(|x|+r_1)\geq \frac{e^{-t_0\LM(B_1^c) }}{2} \, t f(|x|+r_1)  \,.
\end{align*}
The result follows by using the shift property of $f$.
\qed

Only by a mild modification of the proof of Lemma~\ref{lem:low_general_Rd}, 
in one dimension one obtains a one-sided version of that observation.

\begin{lemma}\label{lem:low_general_R1}
Let $d=1$. Assume that a L{\'e}vy measure $\LM(dx)$  is such that
for some $r_0\geq 1$, 
$$\LM(dx) \geq f_+(x)dx \,,\qquad x>r_0\,,$$
where 
$f_+\colon (0,\infty)\to [0,\infty)$
is a non-increasing function.
There is an absolute constant $c$ such that
given $t_0>0$ there exists a constant $\tilde{c}>0$
so that 
$$
p(t,x+t\drf)\geq  \tilde{c}\, tf_+(x+r_1)\,,
$$
holds for each $t\in (0,t_0]$ 
such that
the transition density of $Y_t$ exists,
and almost every $x\geq r_0+r_1$, where $h^{-1}(1/(2ct_0))$.
\end{lemma}

\begin{remark}\label{rem:f_shift}
If 
the function $f$ 
in
Lemma~\ref{lem:low_general_Rd}
satisfies
\eqref{ineq:f_shift}
for some $c_0\geq 1$,
then $f(|x|+r_1)$ may be replaced by  $f(|x|)$.
Similarly,
in Lemma~\ref{lem:low_general_R1}
we would get $f_+(x)$.
\end{remark}

We will now see how the statement of Lemma~\ref{lem:low_general_Rd} and~\ref{lem:low_general_R1} may be improved when further assumptions on the process $Y$
are imposed.
The main ingredient of the proof is \cite[Lemma~3.7]{TGKS-2019}.

\begin{proposition}\label{prop:low_C3_Rd}
Assume that $\Cc$ holds and
a L{\'e}vy measure $\LM(dx)$ satisfies
$$
\LM(dx)\geq f(|x|)dx\,,
$$
where 
$f \colon (0,\infty)\to [0,\infty)$
is a non-increasing function.
Then for all $T,\theta>0$
there is a constant $\tilde{c}=\tilde{c}(d,\alpha_3,c_3,T_3,T,h,\theta)\in (0,1]$ 
such that for every $0<t<T$
there exists $|y_t|\leq (1/\tilde{c})\, h^{-1}(1/t)$
so that for all $|x|\geq \theta h^{-1}(1/t)$,
\begin{align*}
p(t,x+y_t+t\drf_{[h^{-1}(1/t)]}) \geq \tilde{c} \,t f(|x|)\,.
\end{align*}
The statement holds with 
$T=1/h(T_3/c)$, where $c=c(d,\alpha_3,c_3)$,
and $\tilde{c}=\tilde{c}(d,\alpha_3,c_3, \theta)$.
\end{proposition}

\pf 
{\it Step 1.}
We decompose the process $Y$.
For $\lambda>0$ consider 
$$
\LM_{1.\lambda}(dx):= \LM(dx)-\LM_{2.\lambda}(dx)\,,\qquad  
\LM_{2.\lambda}(dx):=\frac12 
\LM|_{B_{\lambda}^c}(dx)\,,
$$
and
$$
\xi_{\lambda}:=\int_{\Rd}z \left(\ind_{|z|<\lambda} - \ind_{|z|<1}\right)\LM_{2.\lambda}(dz)
=- \int_{|z|<1}z  \,\LM_{2.\lambda}(dz)\,.
$$
Let $Z^{1.\lambda}$ and $Z^{2.\lambda}$  be L{\'e}vy processes corresponding to
$(A,N_{1.\lambda},\drf+ \xi_{\lambda})$
and $(0,N_{2.\lambda},-\xi_{\lambda})$.
We write $\LCh_{1.\lambda}$, $\LCh_{2.\lambda}$
for their characteristic exponents;
$h_{1.\lambda}$, $h_{2.\lambda}$
for  their concentration functions
and
$\drf^{1.\lambda}_r$, $b^{2.\lambda}_r$
for
quantities defined in
\eqref{def:br}.
We collect further properties:\\

\begin{itemize}
\item[(i)]  $Z_t^{2.\lambda}$ is a compound Poisson process,
we denote its the distribution by $P_t^{2.\lambda}(dz)$,
$$\LCh_{2.\lambda}(x)=-\int_{\Rd} \left(e^{i\left<x,z\right>}-1\right)\!N_{2.\lambda}(dz)\,.$$

\item[(ii)] $Z^{1.\lambda}$
satisfies the condition $\Cc$
of \cite{TGKS-2019} with $\alpha_3$, $T_3$ and $c_3/2$, since
for every $\lambda>0$,
$$
\frac12 \LM(dx)\leq N_{1.\lambda}(dx)\leq \LM(dx)\,.
$$

\item[(iii)]
Due to $\LCh=\LCh_{1.\lambda}+\LCh_{2.\lambda}$ we have for $t>0$, $x\in\Rd$,
\begin{align*}
p(t,x)=\int_{\Rd} p_{1.\lambda}(t,x-z)P_t^{2.\lambda}(dz)\,.
\end{align*}

\item[(iv)]
For all $r,u>0$,
$$
\frac12 h(r)\leq h_{1.\lambda}(r)\leq h(r)
\quad \mbox{and}\quad 
h^{-1}(2u) \leq h_{1.\lambda}^{-1}(u)\leq h^{-1}(u)\,.
$$
Thus, $|x|\geq \theta h^{-1}(1/t)$
implies $|x|\geq \theta h_{1.\lambda}^{-1}(1/t)$,
while $t<1/h(T_3/c)$ gives $t<1/h_{1.\lambda}(T_3/c)$.\\

\item[(v)]
By the definition in the first equality,
$$
\drf^{1.\lambda}_{\lambda}=\drf+\xi_{\lambda}
+\int_{\Rd} z \left(\ind_{|z|<\lambda} - \ind_{|z|<1}\right)N_{1.\lambda}(dz)
=\drf_{\lambda}\,.
$$
Then by \cite[(8)]{TGKS-2019} and \cite[Lemma~2.1]{TGKS-2019} for all $\varepsilon\in (0,1]$ and $r>0$  satisfying $\lambda \leq \varepsilon \, r $,
\begin{align*}
|\drf^{1.\lambda}_{\lambda/ \varepsilon}-\drf_{r}|
&\leq 
|\drf^{1.\lambda}_{\lambda/\varepsilon}-\drf^{1.\lambda}_{\lambda}|
+ |\drf_{\lambda}-\drf_{r}|\\
&\leq (\lambda/\varepsilon)\,h_{1.\lambda}(\lambda)
+ r h(\lambda)
\leq 2 r h(\lambda)
\leq 4 r  \varepsilon^{-2} h_{1.\lambda}(\lambda/\varepsilon)\,.
\end{align*}

\item[(vi)] 
If $|x|\geq 2 \lambda$ and 
$z\in B^{(x,\lambda)}$,
then 
$$|z|\leq \lambda \leq  |x-z|\leq |x|\,,$$
where for $x\neq 0$ and $\lambda>0$  
$$
B^{(x,\lambda)}:= B(\lambda x/(2|x|), \lambda/2)\,.
$$ 
\end{itemize}

\noindent
{\it Step 2.}
Now, by 
\cite[Lemma~3.7]{TGKS-2019}
applied
to $Z^{1.\lambda}$
we have
that
there 
is
a constant
$
c=c(d,\alpha_3,c_3)\in [1,\infty)
$
such that
for every
$0<t<1/h_{1.\lambda}(T_3/c)$
there exists $|x_t|\leq c\,h_{1.\lambda}^{-1}(1/t)$
so that for every
$|z|\leq (1/c)\, h_{1.\lambda}^{-1}(1/t)$,
$$
p_{1.\lambda}(t, z+x_t+ t\drf^{1.\lambda}_{[h_{1.\lambda}^{-1}(1/t)]})
\geq (1/c) \left[h_{1.\lambda}^{-1}(1/t)\right]^{-d}.
$$
Set 
$\varepsilon=(1/c) \land (\theta/2)$,
$\lambda=\varepsilon\, h_{1.\lambda}^{-1}(1/t)$
and 
$r=h^{-1}(1/t)$.
By property (iv) we can use (v)
to get for all $t>0$,
$$
|\drf^{1.\lambda}_{[h_{1.\lambda}^{-1}(1/t)]}- \drf_{[h^{-1}(1/t)]}|
\leq 4 \varepsilon^{-2} t^{-1} h^{-1}(1/t)\,.
$$
By \cite[Lemma~2.1]{TGKS-2019} for all $t>0$,
$$
N_{2.\lambda}(\Rd)=\int_{|z|\geq \lambda} \frac12\LM(dz)
=\int_{|z|\geq \lambda} \LM_{1.\lambda}(dz)
\leq h_{1.\lambda}(\lambda)
\leq \varepsilon^{-2} t^{-1}\,.
$$
Thus,
by property (vi),
for all
$0<t<1/h_{1.\lambda}(T_3/c)$
and $|x|\geq \theta h_{1.\lambda}^{-1}(1/t)$,
with $ y_t=x_t+t\drf^{1.\lambda}_{[h_{1.\lambda}^{-1}(1/t)]}-t \drf_{[h^{-1}(1/t)]}$,
\begin{align*}
p(t,x+y_t + &t \drf_{[h^{-1}(1/t)]})
= \int_{\Rd} p_{1.\lambda}(t,x+y_t+t \drf_{[h^{-1}(1/t)]}-z) P_t^{2.\lambda}(dz)\\
&\geq \int_{\Rd} p_{1.\lambda}(t,x+y_t+t \drf_{[h^{-1}(1/t)]}-z)\, e^{-t \LM_{2.\lambda}(\Rd)}\,t\LM_{2.\lambda}(dz)\\
& \geq \int_{B^{(x,\lambda)}}  p_{1.\lambda}(t,z+y_t+t \drf_{[h^{-1}(1/t)]})\,\frac12 e^{-t \LM_{2.\lambda}(\Rd)}  \, t f(|x-z|)\,dz\\
& \geq \left(\int_{B^{(x,\lambda)}}  p_{1.\lambda}(t,z+y_t+t \drf_{[h^{-1}(1/t)]})\,dz\right) \frac12 e^{-t \LM_{2.\lambda}(\Rd)}  \, t f(|x|)\\
&= \left(\int_{B^{(x,\lambda)}}  p_{1.\lambda}(t,z+x_t+t\drf^{1.\lambda}_{[h_{1.\lambda}^{-1}(1/t)]})\,dz\right) \frac12 e^{-t \LM_{2.\lambda}(\Rd)}  \, t f(|x|)\\
&\geq (1/c)\, \omega_d (\varepsilon/2)^d\,  \frac12 e^{-1/\varepsilon^2}  t f(|x|)\,.
\end{align*}
By implications in property (iv) and
$$
|y_t|\leq |x_t|+|t\drf^{1.\lambda}_{[h_{1.\lambda}^{-1}(1/t)]}-t \drf_{[h^{-1}(1/t)]}|
\leq \left( c+ 4 \varepsilon^{-2}\right) h^{-1}(1/t)\,,
$$
the latter ends the proof of the last sentence of the proposition.

\noindent
{\it Step 3.}
It remains to consider the case that $T_3<\infty$ and $h(T_3/c)\leq t<T$.
By 
\cite[Lemma~3.7]{TGKS-2019}
applied to $Y$
we have
that
there 
is
a constant
$
c=c(d,\alpha_3,c_3,T_3,T,h)\in [1,\infty)
$
such that
for every
$0<u<T$
there exists $|x_u|\leq c\,h^{-1}(1/u)$
so that for every
$|z|\leq (1/c)\, h^{-1}(1/u)$,
$$
p(u, z+x_u+ u\,\drf_{[h^{-1}(1/u)]})
\geq (1/c) \left[h^{-1}(1/u)\right]^{-d}.
$$
We have already shown that
there 
is a constant $\tilde{c}=\tilde{c}(d,\alpha_3,c_3,\theta)\in (0,1]$ 
such that for every $0<u<1/h(T_3/c)$
there exists $|y_u|\leq (1/\tilde{c})\, h^{-1}(1/u)$
so that for all $|x-z|\geq \varepsilon\, h^{-1}(1/u)$,
\begin{align*}
p(u,x-z+y_u+u\,\drf_{[h^{-1}(1/u)]}) \geq \tilde{c} \,u f(|x-z|)\,.
\end{align*}
Let $t_0=(1/2) / h(T_3/c)$ and $\lambda_0=\varepsilon\,h^{-1}(1/t_0)$.
Since $t_0<2t_0=1/h(T_3/c)\leq t<T$,
by property~(vi),
for  $y_t= y_{t_0}+t_0 \drf_{[h^{-1}(1/t_0)]}+x_{t-t_0}+(t-t_0)\drf_{[h^{-1}(1/(t-t_0))]} - t \drf_{[h^{-1}(1/t)]}$ 
and all $|x|\geq \theta h^{-1}(1/t)$
we have 
\begin{align*}
p(t,x+&y_t+t \drf_{[h^{-1}(1/t)]})
=\int_{\Rd} p(t-t_0,z)\, p(t_0, x+y_t+t \drf_{[h^{-1}(1/t)]}-z)\,dz\\
&\geq \int_{B^{(x,\lambda_0)}} p(t-t_0,z+x_{t-t_0}+(t-t_0)\drf_{[h^{-1}(1/(t-t_0))]})\, p(t_0, x-z+y_{t_0}+t_0 \drf_{[h^{-1}(1/t_0)]})\,dz\\
&\geq  \left( \int_{B^{(x,\lambda_0)}}p(t-t_0,z+x_{t-t_0}+(t-t_0)\drf_{[h^{-1}(1/(t-t_0))]})\,dz \right)  \tilde{c} \,t_0 f(|x|)\\
&\geq (1/c) \left[ h^{-1}(1/(t-t_0))\right]^{-d}
\omega_d 
(\varepsilon/2)^d \left[ h^{-1}(1/t_0)\right]^{d}
 \tilde{c} \,t_0 f(|x|)\\
&\geq (1/c)\, \omega_d (\varepsilon/2)^d 
\left(\frac{h^{-1}(1/t_0)}{h^{-1}(1/T)}\right)^d 
\tilde{c}\, \frac{t_0}{T} \,t
f(|x|)\,.
\end{align*}
We also have $|x_{t-t_0}|\leq c\, h^{-1}(1/(t-t_0))\leq c\, h^{-1}(1/t)$, $|y_{t_0}|\leq (1/\tilde{c})\,h^{-1}(1/t_0)\leq (1/\tilde{c})\,h^{-1}(1/t)$
and by \cite[(8)]{TGKS-2019},
\begin{align*}
&|t_0 \drf_{[h^{-1}(1/t_0)]}+(t-t_0) \drf_{[h^{-1}(1/(t-t_0))]}-t \drf_{[h^{-1}(1/t)]}|\\
&=|t_0 (\drf_{[h^{-1}(1/t_0)]}-\drf_{[h^{-1}(1/t)]})+(t-t_0)(\drf_{[h^{-1}(1/(t-t_0))]}-\drf_{[h^{-1}(1/t)]})|\\
&\leq h^{-1}(1/t) \big[ t_0 h(h^{-1}(1/t_0))+(t-t_0) h(h^{-1}(1/(t-t_0))) \big]
\leq 2 h^{-1}(1/t)\,.
\end{align*}
That gives $|y_t|\leq (c +(1/\tilde{c})+2)h^{-1}(1/t)$ and ends the proof.
\qed

Again, a one-sided one-dimensional version is valid.

\begin{proposition}\label{prop:low_C3_R1}
Let $d=1$. Assume that $\Cc$ holds and
a L{\'e}vy measure $\LM(dx)$ satisfies
$$
\LM(dx) \geq f_+(x)dx \,,\qquad x>0\,,
$$
where 
$f_+ \colon (0,\infty)\to [0,\infty)$
is a non-increasing function.
Then for all $T,\theta>0$
there is a constant $\tilde{c}=\tilde{c}(d,\alpha_3,c_3,T_3,T,h,\theta)\in (0,1]$ 
such that for every $0<t<T$
there exists $|y_t|\leq (1/\tilde{c})\, h^{-1}(1/t)$
so that for all $x\geq \theta h^{-1}(1/t)$,
\begin{align*}
p(t,x+y_t+t\drf_{[h^{-1}(1/t)]}) \geq \tilde{c} \,t f_+(x)\,.
\end{align*}
The statement holds with 
$T=1/h(T_3/c)$, where $c=c(d,\alpha_3,c_3)$,
and $\tilde{c}=\tilde{c}(d,\alpha_3,c_3, \theta)$.
\end{proposition}

In the next result we concentrate on a L{\'e}vy process $Y$ with zero Gaussian part and such that some essential symmetric jumps  occur.
This allows us to remove the anonymous shift $y_t$ that appears in Theorem~\ref{prop:low_C3_Rd}.
The key ingredient of the proof is \cite[Theorem~5.2]{TGKS-2019}.

\begin{theorem}\label{thm:low_C3_sym_Rd}
Assume 
that $\Cc$ holds, $A=0$ and
a L{\'e}vy measure $\LM(dx)$ satisfies
$$
\LM(dx)\geq f(|x|)dx\,,
$$
where 
$f \colon (0,\infty)\to [0,\infty)$
is a non-increasing function.
Suppose
there is
 $a_1\in(0,1]$ 
such that
$$
a_1 \,\nu_s (dx) \leq \LM(dx)\,,
$$
and $a_2 \in [1,\infty)$ such that for every $|x|>1/T_3$,
$$
{\rm Re}[\LCh (x)] \leq a_2 \,{\rm Re}[\LCh_s (x)]\,. 
$$
Here $\nu_s(dx)$ is a symmetric L{\'e}vy measure; $\LCh_s(x)$ and $h_s(r)$ correspond to $(0,\nu_s,0)$.
Then
for all $T, \theta_1, \theta_2>0$ 
there 
is
a constant
$\tilde{c}=\tilde{c}(d,\alpha_3,c_3,T_3,a_1,a_2,\nu_s,T,\theta_1,\theta_2)\in (0,1]$  
such that for all
$0<t<T$,
$|y|\leq \theta_1 h_s^{-1}(1/t)$
and
 $|x|\geq \theta_2 h_s^{-1}(1/t)$,
$$
p(t,x+y+t \drf_{[h_s^{-1}(1/t)]}) \geq \tilde{c} \,t f(|x|)\,.
$$
The statement holds with 
$T=1/h_s(T_3/c)$, where $c=c(d,\alpha_3,c_3,a_2)$,
and $\tilde{c}=\tilde{c}(d,\alpha_3,c_3,T_3,a_1,a_2,\nu_s,\theta_1,\theta_2)$.
If $T_3=\infty$, we also have $\tilde{c}>0$.
\end{theorem}

\pf
{\it Step 1.} We decompose $Y$ into $Z^{1.\lambda}$ and $Z^{2.\lambda}$ like in {\it Step 1.} of the proof of Proposition~\ref{prop:low_C3_Rd}, but
recall that $A=0$.
Under present constraints we additionally have
\begin{itemize}
\item[(vii)]
For $\lambda>0$,
$$
\frac{a_1}{2} \nu_s (dx)\leq \LM_{1.\lambda}(dx)\,,
$$
and if $|x|>1/T_3$,
$$
{\rm Re} [\LCh_{1.\lambda}(x)]\leq a_2 \,{\rm Re}[\LCh_s (x)]\,.
$$
Thus $Z^{1.\lambda}$ satisfies the assumptions of 
\cite[Theorem~5.2]{TGKS-2019}.\\

\item[(viii)] For $\lambda>0$ and $r<T_3$,
$$
h_{1.\lambda}(r)\leq h(r) \leq  a_2 (c_d/c_3) h_s(r)\,,
$$
holds with $c_d=16(1+2d)$. See \cite[Remark~4.1(iii)]{TGKS-2019}. \\
\end{itemize}

\noindent
{\it Step 2.}
Let
$\varepsilon=1\land (\theta_2/2)$ and
$\theta_3=4 \varepsilon^{-2} a_2 (c_d/c_3)$.
Now, by the {\it Claim}  justified in the proof of \cite[Theorem~5.2]{TGKS-2019}
applied to $Z^{1.\lambda}$ and $\theta=\theta_1+\theta_2/2+\theta_3$,
there are $c=c(d,\alpha_3,c_3,a_2)\in [1,\infty)$
and $\tilde{c}=\tilde{c}(d,\alpha_3,c_3,T_3,a_1,a_2,\nu_s,\theta_1,\theta_2)\in (0,1]$ such that for all
$0<t<1/h_s(T_3/c)$ and $|w|\leq (\varepsilon+\theta_1+\theta_3) h_s^{-1}(1/t)$,
$$
p_{1.\lambda}(t,w+t \drf^{1.\lambda}_{[h_s^{-1}(1/t)]})\geq \tilde{c} \left[h_s^{-1}(1/t)\right]^{-d}.
$$
If $T_3=\infty$, we also have $\tilde{c}>0$.
Set
$\lambda= \varepsilon\, h_s^{-1}(1/t)$
and
$r=h_s^{-1}(1/t)$.
By properties (v) and~(viii) we have for all $0<t<1/h_s(T_3)$,
$$
|\drf_{[h_s^{-1}(1/t)]}-\drf^{1.\lambda}_{[h_s^{-1}(1/t)]}|
\leq 4r \varepsilon^{-2} h_{1.\lambda}(r)
\leq 4r \varepsilon^{-2} a_2 (c_d/c_3)  h_s(r)
= \theta_3 \,t^{-1} h_s^{-1}(1/t)\,.
$$
Further, by \cite[Lemma~2.1]{TGKS-2019} we have for all $0<t<1/h_s(T_3/\varepsilon)$,
$$
N_{2.\lambda}(\Rd)=\int_{|z|\geq \lambda} \frac12\LM(dz)
=\int_{|z|\geq \lambda} \LM_{1.\lambda}(dz)
\leq h_{1.\lambda}(\lambda)
\leq a_2 (c_d/c_3) h_s(\lambda)
\leq (\theta_3/4)\, t^{-1}\,.
$$
Thus,
by property (vi),
for all $0<t<1/h_s(T_3/c)$, 
$|y|\leq \theta_1 h_s^{-1}(1/t)$
and
$|x|\geq \theta_2 h_s^{-1}(1/t)$,
\begin{align*}
p(t,x+ & y+t \drf_{[h_s^{-1}(1/t)]})
= \int_{\Rd} p_{1.\lambda}(t,x+y+t \drf_{[h_s^{-1}(1/t)]}-z) P_t^{2.\lambda}(dz)\\
&\geq \int_{\Rd} p_{1.\lambda}(t,x+y+t \drf_{[h_s^{-1}(1/t)]}-z)\, e^{-t \LM_{2.\lambda}(\Rd)}\,t\LM_{2.\lambda}(dz)\\
& \geq \int_{B^{(x,\lambda)}}  p_{1.\lambda}(t,z+y+t \drf_{[h_s^{-1}(1/t)]})\,\frac12 e^{-t \LM_{2.\lambda}(\Rd)}  \, t f(|x-z|)\,dz\\
& \geq \left(\int_{B^{(x,\lambda)}}  p_{1.\lambda}(t,z+y+t \drf_{[h_s^{-1}(1/t)]})\,dz\right) \frac12 e^{-t \LM_{2.\lambda}(\Rd)}  \, t f(|x|)\\
&= \left(\int_{B^{(x,\lambda)}}  p_{1.\lambda}(t,z+y+t \drf_{[h_s^{-1}(1/t)]}-t \drf^{1.\lambda}_{[h_s^{-1}(1/t)]}+t \drf^{1.\lambda}_{[h_s^{-1}(1/t)]})\,dz\right) \frac12 e^{-t \LM_{2.\lambda}(\Rd)}  \, t f(|x|)\\
&\geq \tilde{c}\, \omega_d (\varepsilon/2)^d \,\frac12 e^{-\theta_3/4}\,  t f(|x|)\,,
\end{align*}
since 
$
|z+y+t \drf_{[h_s^{-1}(1/t)]}-t \drf^{1.\lambda}_{[h_s^{-1}(1/t)]}|\leq (\varepsilon+\theta_1+\theta_3) h_s^{-1}(1/t)
$.
This ends the proof of the last two sentences of the proposition.

\noindent
{\it Step 3.}
It remains to consider the case that $T_3<\infty$ and $h_s(T_3/c)\leq t<T$.
Let
$t_0=(1/2) / h_s(T_3/c)$
and
$
\theta_4= a_2 (c_d/c_3) [h_s^{-1}(1/T)/h_s^{-1}(1/t_0)]\, (T/t_0)
$.
By 
\cite[Theorem~5.2]{TGKS-2019}
applied to $Y$
we have
that
there 
is
a constant
$
\tilde{c}_1=\tilde{c}_1(d,\alpha_3,c_3,T_3,a_1,a_2,\nu_s,T,\theta_1,\theta_2)\in (0,1]
$
such that
for all
$0<u<T$
and
$|w|\leq (\varepsilon+\theta_1+\theta_4)\, h_s^{-1}(1/u)$,
$$
p(u, w+ u\,\drf_{[h_s^{-1}(1/u)]})
\geq \tilde{c}_1 \left[h_s^{-1}(1/u)\right]^{-d}.
$$
We have already shown that
there 
is a constant $\tilde{c}_2=\tilde{c}_2(d,\alpha_3,c_3,T_3,a_1,a_2,\nu_s,\theta_2)\in (0,1]$ 
such that for all $0<u<1/h_s(T_3/c)$
and $|x-z|\geq \varepsilon\, h_s^{-1}(1/u)$,
\begin{align*}
p(u,x-z+u\,\drf_{[h_s^{-1}(1/u)]}) \geq \tilde{c}_2 \,u f(|x-z|)\,.
\end{align*}
By \cite[(8)]{TGKS-2019}
and property (viii),
\begin{align*}
&|t \drf_{[h_s^{-1}(1/t)]}-t_0 \drf_{[h_s^{-1}(1/t_0)]}-(t-t_0) \drf_{[h_s^{-1}(1/(t-t_0))]}|\\
&=|(t-t_0)(\drf_{[h_s^{-1}(1/t)]}-\drf_{[h_s^{-1}(1/(t-t_0))]})
+t_0 (\drf_{[h_s^{-1}(1/t)]}-\drf_{[h_s^{-1}(1/t_0)]})|\\
&\leq h_s^{-1}(1/t) \big[ (t-t_0) h(h_s^{-1}(1/(t-t_0))+t_0 h(h_s^{-1}(1/t_0))) \big]\\
&\leq h_s^{-1}(1/T)\, t h(h_s^{-1}(1/t_0))
\leq 
\theta_4\, h_s^{-1}(1/(t-t_0))\,.
\end{align*}
Set $\lambda_0=\varepsilon\,h_s^{-1}(1/t_0)$.
Since $t_0<2t_0=1/h_s(T_3/c)\leq t<T$,
by property~(vi),
we have  for all $|x|\geq \theta_2 h_s^{-1}(1/t)$,
\begin{align*}
p(t,x+&y+t \drf_{[h_s^{-1}(1/t)]})
=\int_{\Rd} p(t-t_0,z)\, p(t_0, x+y+t \drf_{[h_s^{-1}(1/t)]}-z)\,dz\\
&\geq \int_{B^{(x,\lambda_0)}} p(t-t_0,z+y+t \drf_{[h_s^{-1}(1/t)]}-t_0 \drf_{[h_s^{-1}(1/t_0)]})\, p(t_0, x-z+t_0 \drf_{[h_s^{-1}(1/t_0)]})\,dz\\
&\geq  \left( \int_{B^{(x,\lambda_0)}}p(t-t_0,z+y+t \drf_{[h_s^{-1}(1/t)]}-t_0 \drf_{[h_s^{-1}(1/t_0)]})\,dz \right)  \tilde{c}_2 \,t_0 f(|x|)\\
&=  \left( \int_{B^{(x,\lambda_0)}}p(t-t_0,z+\bar{y}+(t-t_0) \drf_{[h_s^{-1}(1/(t-t_0))]})\,dz \right)  \tilde{c}_2 \,t_0 f(|x|)\\
&\geq \tilde{c}_1 \left[ h_s^{-1}(1/(t-t_0))\right]^{-d}
\omega_d 
(\varepsilon/2)^d \left[ h_s^{-1}(1/t_0)\right]^{d}
 \tilde{c}_2 \,t_0 f(|x|)\\
&\geq \tilde{c}_1\, \omega_d (\varepsilon/2)^d 
\left(\frac{h_s^{-1}(1/t_0)}{h_s^{-1}(1/T)}\right)^d 
\tilde{c}_2\, \frac{t_0}{T} \,t
f(|x|)\,.
\end{align*}
since $|z|\leq \varepsilon \,h_s^{-1}(1/t_0)\leq \varepsilon \,h_s^{-1}(1/(t-t_0))$,
$|y|\leq \theta_1 h_s^{-1}(1/t_0)\leq \theta_1 \,h_s^{-1}(1/(t-t_0))$
and
$$
\bar{y}:=y+t \drf_{[h_s^{-1}(1/t)]}-t_0 \drf_{[h_s^{-1}(1/t_0)]}-(t-t_0) \drf_{[h_s^{-1}(1/(t-t_0))]}\,,
$$
satisfies $|\bar{y}|\leq (\theta_1+\theta_4)\, h_s^{-1}(1/(t-t_0))$.
This ends the proof.
\qed

The next result
resembles
Theorem~\ref{thm:low_C3_sym_Rd},
but instead of using a symmetric measure $\nu_s$ we adopt arbitrary measure $\nu$ with additional restriction in the scaling condition of $\Cc$.
We skip the proof, which is the same as that of
Theorem~\ref{thm:low_C3_sym_Rd}
with   the
only change in {\it Step 2.} where the main component becomes this time
\cite[Theorem~5.2]{TGKS-2019}.

\begin{theorem}\label{thm:low_C3_scl1_Rd}
Assume 
that $\Cc$ holds with $\alpha_3\geq 1$ and $A=0$.  Let a L{\'e}vy measure $\LM(dx)$ satisfy
$$
\LM(dx)\geq f(|x|)dx\,,
$$
where 
$f \colon (0,\infty)\to [0,\infty)$
is a non-increasing function.
Suppose
there is
 $a_1\in(0,1]$ 
such that
$$
a_1 \,\nu (dx) \leq \LM(dx)\,,
$$
and $a_2 \in [1,\infty)$ such that for every $|x|>1/T_3$,
$$
{\rm Re}[\LCh (x)] \leq a_2 \,{\rm Re}[\LCh_{\nu} (x)]\,. 
$$
Here $\nu(dx)$ is a L{\'e}vy measure; $\LCh_{\nu}(x)$ and $h_{\nu}(r)$ correspond to $(0,\nu,0)$.
Then
for all $T, \theta_1, \theta_2>0$ 
there 
is
a constant
$\tilde{c}=\tilde{c}(d,\alpha_3,c_3,T_3,a_1,a_2,\nu,T,\theta_1,\theta_2)\in (0,1]$  
such that for all
$0<t<T$,
$|y|\leq \theta_1 h_{\nu}^{-1}(1/t)$
and
 $|x|\geq \theta_2 h_{\nu}^{-1}(1/t)$,
$$
p(t,x+y+t \drf_{[h_{\nu}^{-1}(1/t)]}) \geq \tilde{c} \,t f(|x|)\,.
$$
The statement holds with 
$T=1/h_{\nu}(T_3/c)$, where $c=c(d,\alpha_3,c_3,a_2)$,
and $\tilde{c}=\tilde{c}(d,\alpha_3,c_3,T_3,a_1,a_2,\nu,\theta_1,\theta_2)$.
If $T_3=\infty$, we also have $\tilde{c}>0$.
\end{theorem}

\begin{remark}
It is important for applications
that the results of
Theorem~\ref{thm:low_C3_sym_Rd}
and Theorem~\ref{thm:low_C3_scl1_Rd}
are uniform for the whole class of L{\'e}vy processes $Y$ as long as certain parameters do not change.
\end{remark}

The results of Theorem~\ref{thm:low_C3_sym_Rd}
and Theorem~\ref{thm:low_C3_scl1_Rd}
have their one-sided one-dimensional versions.
\begin{proposition}\label{prop:R1}
The following variant of Theorem~\ref{thm:low_C3_sym_Rd}
is true
(resp. Theorem~\ref{thm:low_C3_scl1_Rd})
if $d=1$: 
we replace the assumption $\LM(dx)\geq f(|x|)dx$ with
$$
\LM(dx)\geq f_+(x)dx\,,\qquad x>0\,,
$$
where $f_+ \colon (0,\infty)\to [0,\infty)$ is a non-increasing function,
further the condition $|x|\geq \theta_2 h_s^{-1}(1/t)$
(resp. $|x|\geq \theta_2 h_{\nu}^{-1}(1/t)$)
is replaced by
$$
x\geq \theta_2 h_s^{-1}(1/t)
\qquad
\,
(\mbox{resp. } x\geq \theta_2 h_{\nu}^{-1}(1/t))\,,
$$
and then the final inequality by
$$
p(t,x+y+t \drf_{[h_s^{-1}(1/t)]}) \geq \tilde{c} \,t f_+(x)
\qquad
\,
(\mbox{resp. } p(t,x+y+t \drf_{[h_{\nu}^{-1}(1/t)]}) \geq \tilde{c} \,t f_+(x))\,.
$$
\end{proposition}

According to Proposition~\ref{prop:R1} 
by the one one-sided one-dimensional version of Theorem~\ref{thm:low_C3_scl1_Rd}
with $\nu=\LM$
the following holds.

\begin{corollary}\label{cor:low_C3_scl1_R1}
Let $d=1$.
Assume that $\Cc$
holds with $\alpha_3\geq 1$ and $A=0$.
Suppose
that a L{\'e}vy measure $\LM(dx)$  is such that
$$
\LM(dx) \geq  f_+(x)dx\,,\qquad x>0\,,
$$
where 
$f_+\colon (0,\infty)\to [0,\infty)$
is a non-increasing function.
Then for all $T,\theta_1, \theta_2 >0$ there 
is a constant $\tilde{c}=\tilde{c}(\alpha_3,c_3,T_3,\LM,T,\theta_1,\theta_2)$
such that for all
$0<t<T$,
$|y|\leq \theta_1 h^{-1}(1/t)$
and
 $|x|\geq \theta_2 h^{-1}(1/t)$,
\begin{align*}
p(t,x+y+t\drf_{[h^{-1}(1/t)]}) \geq \tilde{c} \,t f_+(x)\,, 
\end{align*}
The statement holds with 
$T=1/h(T_3/c)$, where $c=c(d,\alpha_3,c_3,a_2)$,
and $\tilde{c}=\tilde{c}(d,\alpha_3,c_3,T_3,\LM,\theta_1,\theta_2)$.
If $T_3=\infty$, we also have $\tilde{c}>0$.
\end{corollary}

\begin{remark}\label{rem:x_neg2}
One-sided one-dimensional versions of Lemma~\ref{lem:low_general_R1}, Proposition~\ref{prop:low_C3_R1} and~\ref{prop:R1}, 
and Corollary~\ref{cor:low_C3_scl1_R1}
for $x<0$ are also valid, cf. Corollary~\ref{cor:one_side} and Remark~\ref{rem:x_neg1}.
\end{remark}

\section{L{\'e}vy measure comparable with a unimodal one}\label{subsec:unim_comp}
\label{sec:cU}

We concentrate on
a L{\'e}vy process $Y$ in $\Rd$
with a generating triplet $(A,\LM,\drf)$, where $A=0$,
$\drf\in\Rd$ and
$\LM(dx)=\gLM(x)dx$
satisfies for some $\lmC \geq 1$,
\begin{equation}\label{e:assumption-j0}
\lmC^{-1}\uLM(|x|)\leq \gLM(x) \leq \lmC \uLM(|x|)\, , \qquad x\in \Rd\, ,
\end{equation}
where $\uLM\colon [0,\infty)\to [0,\infty]$ is a non-increasing function.
In other words, $\LM(dx)$ is comparable with an isotropic unimodal L{\'e}vy measure
$\uLM(|x|)dx$.
We denote 
the corresponding characteristic exponent by $\uLCh$,
and similarly functions $\uK$ and $\uh$.
At this point we refer
the reader to Section~\ref{sec:uni_app}
for auxiliary results on isotropic unimidal L{\'e}vy processes.
Further, we consider the scaling condition for $\uh$
(see \cite[Remark~2.12]{TGKS-2019}): there are 
$\ulah \in(0,2]$, $C_{\uh}\in[1,\infty)$ and $\theta_{\uh}\in(0,\infty]$ such that
\begin{equation}\label{eq:wlsc:h}
 \uh(r)\leq C_{\uh}\lambda^{\ulah } \uh(\lambda r),\qquad \lambda\leq 1,\, r< \theta_{\uh}.
\end{equation} 
\begin{remark}\label{rem:uni_scal_C3}
We stress that under the assumption of  \eqref{e:assumption-j0}
the condition \eqref{eq:wlsc:h} is equivalent with $\Cc$ for the process $Y$,
see \cite[Lemma~3.5]{TGKS-2019}
or more directly \eqref{e:assumption-j} below.
\end{remark}

We  focus on the upper estimates of the density and its derivatives. To this end we will use Theorem~\ref{thm:KP}.
We simplify the notation by introducing
for  $t>0$ and $x\in \Rd$  {\it the bound function},
\begin{align}\label{def:bound_function}
\rr_t(x)=\left( [\uh^{-1}(1/t)]^{-d}\land \frac{t\uK(|x|)}{|x|^{d}}\right).
\end{align}

\begin{theorem}\label{p:upperestonp}
Assume that the L{\'e}vy measure $\LM(dx)$ satisfies
\eqref{e:assumption-j0} and that \eqref{eq:wlsc:h} holds for~$\uh$.
Let $\bbbeta\in \mathbb{N}_0^d$,
$c_d=16(1+2d)$ and $\theta >0$. 
Then 
for all $t< 1/[ 2 c_d \lmC \uh(\theta_{\uh})]$,
$|y|\leq \theta \uh^{-1}(1/t)$
 and $x\in\Rd$,
\begin{align}\label{ineq:Grad_0}
|\partial_x^{\bbbeta}\, p(t,x+y+ t \drf_{[\uh^{-1}(1/t)]})|
\leq 
\ubC \left[ \uh^{-1}(1/t) \right]^{-|\bbbeta|}
\rr_t (x)\,,
\end{align}
holds 
with $\ubC=\ubC(d,\bbbeta, \ulah,C_{\uh},\lmC,\theta)$.
\end{theorem}

\pf
From \eqref{e:assumption-j0} it follows that $\lmC^{-1} \uh (r)\leq h(r) \leq \lmC \uh (r)$, $r>0$,
and combining it with \eqref{ineq:comp_unimod} and 
\eqref{ineq:comp_TJ} gives
for every $x\neq 0$ and $r=|x|$,
\begin{equation}\label{e:assumption-j}
((c_d/2)\pi^2 \lmC)^{-1}\, \uh(1/r) \leq {\rm Re}[ \LCh(x)] \leq \LCh^*(r)\leq (2 \lmC)\, \uh(1/r)\,. 
\end{equation}
Together with \eqref{eq:wlsc:h} and Lemma~\ref{lem:inverse_c}
this gives
\begin{align*}
\left[\uh^{-1}(u/(2\lmC))\right]^{-1} \leq \LCh^{-1}(u)\leq \left[\uh^{-1}((c_d/2)\pi^2 \lmC u)\right]^{-1}\,,\qquad u>0\,.  
\end{align*}
Applying \cite[Lemma~2.3]{TGKS-2019} we obtain
\begin{align}\label{ineq:inverse_equiv}
 \frac{[C_{\uh} 2\lmC]^{-1/\ulah}}{\uh^{-1}(u)} \leq \LCh^{-1}(u)\leq \frac{[C_{\uh} (c_d/2)\pi^2 \lmC]^{1/\ulah}}{\uh^{-1}(u)}\,,\qquad u> 2\lmC \uh(\theta_{\uh})\,.
\end{align}

\noindent
We verify assumptions of Theorem~\ref{thm:KP}. We define a decreasing function $f(s)=M_1 s^{-d}\uK(s)$ and
we have $f(s/2)\leq 2^{d+2}f(s)$ (see 
Lemma~\ref{lem:basic_prop_K_h+} and
\cite[Lemma~2.1]{TGKS-2019}),
where the constant
$M_1=M_1(d,\lmC)$
is chosen such that
\eqref{KP_as1} is satisfied  with $f$ and $\gamma=d$. This is indeed possible since for a Borel set $\setA \subseteq \Rd$, $x\in \setA$, we have
\begin{align*}
\LM(\setA)&\leq \lmC \int_{\setA} \uLM(|z|)\,dz \leq \lmC\, \uLM(\delta(\setA))|\setA|\leq \lmC\, \uLM (\delta(\setA))\,|B(x,{\rm diam}(\setA))|\\
&= \left(\lmC\, 
\omega_d/d
\right) \uLM (\delta(\setA)) \big[{\rm diam}(\setA) \big]^d\,,
\end{align*}
and $\uLM(r)\leq c(d) f(r)$.
Furthermore, since $f$ is decreasing $f(s\vee |y|-|y|/2)\leq f((s\vee|y|)/2)\leq f(s/2)$ for $s>0$, $y\in\Rd$.
Consequently, by 
\eqref{ineq:comp_TJ}
 we have for $s,r>0$,
\begin{align*}
\int_{|y|>r} f(s\vee|y|-|y|/2)\,\LM(dy)
\leq  f(s/2)h(r)
\leq  2^{d+1} \frac{c_d}{M_1} f(s) \LCh^*(1/r)\,.
\end{align*}
Thus \eqref{KP_as2} holds for $f$ with $M_2=M_2(d,\lmC)$. 
Now,  
using \eqref{e:assumption-j} and \eqref{eq:wlsc:h} we get that
$\Cc$ holds for $\LCh$
with $\alpha_3=\ulah$, $c_3=c_3(d,C_{\uh},\lmC)$ and $T_3=\theta_{\uh}$. 
By \cite[Theorem 3.1]{TGKS-2019} and Lemma~\ref{lem:eqiv_h_LCh}
we obtain
\eqref{KP_as8} 
with $M_5=M_5(d,\ulah,C_{\uh},\lmC)$ and $T=1/(2\uh(\theta_{\uh}))$.
Then by Theorem~\ref{thm:KP} with $m=d+2+|\bbbeta|$ and $k=d+2$ there exists a constant 
$c=c(d,\bbbeta,\ulah,C_{\uh},\lmC)$ (which may change from line to line in what follows) such that
\begin{align*}
|\partial_{x}^{\bbbeta} p\left(t, x+t \drf_{\left[1/\LCh^{-1}(1/t)\right]}\right)|
&\leq c \left[\LCh^{-1}(1/t) \right]^{|\bbbeta|} \left(\left[\LCh^{-1}(1/t) \right]^{d}\land \left(tf(|x|/4)  + \frac{\left[\LCh^{-1}(1/t)\right]^d}{\left[ 1+ |x|\LCh^{-1}(1/t)\right]^{d+2}}\right) \right)\\
&\leq c \left[\LCh^{-1}(1/t) \right]^{|\bbbeta|} \left(\left[\LCh^{-1}(1/t) \right]^{d}\land \left(tf(|x|)  + \left[|x|\LCh^{-1}(1/t)\right]^{-2}|x|^{-d} \right) \right).
\end{align*}
Here and below $t\in (0,T/(c_d \lmC )]$.
By \eqref{ineq:inverse_equiv}
\begin{align*}
|\partial_{x}^{\bbbeta} p\left(t, x+t\drf_{\left[1/\LCh^{-1}(1/t) \right]}\right)|
\leq c\left[\uh^{-1}(1/t) \right]^{-|\bbbeta|} \left(\left[\uh^{-1}(1/t) \right]^{-d}\land \left(tf(|x|)  + \left[\uh^{-1}(1/t)/|x|\right]^{2} |x|^{-d} \right) \right).
\end{align*}
If $|x|\leq \uh^{-1}(1/t)$,
 then  we bound the above minimum by $\left[\uh^{-1}(1/t) \right]^{-d}$. 
If
$|x|> h^{-1}(1/t)$, 
we proceed as follows: recall that $\lambda^2 \leq \uh(r)/\uh(\lambda r)$, $\lambda\leq 1$, $r>0$. Taking $\lambda=\uh^{-1}(1/t)/(|x|\land \theta_{\uh})$ and $r=|x|\land \theta_{\uh}$
we get
\begin{align*}
\left(\frac{\uh^{-1}(1/t)}{|x|\land \theta_{\uh}}\right)^2\leq  \frac{\uh(|x|\land \theta_{\uh})}{\uh\left(\uh^{-1}(1/t) \right)}= t\uh(|x|\land \theta_{\uh})\,.
\end{align*} 
Then
\begin{align*}
\left(\frac{\uh^{-1}(1/t)}{|x|}\right)^2
&=\left(\frac{\uh^{-1}(1/t)}{|x|\land \theta_{\uh}}\right)^2\frac{(|x|\land \theta_{\uh})^2}{|x|^2}\leq 
t \uh(|x|\land \theta_{\uh}) \frac{(|x|\land \theta_{\uh})^2\uK(|x|\land \theta_{\uh})}{|x|^2\uK(|x|)}\frac{\uK(|x|)}{\uK(|x|\land \theta_{\uh})}\\
&\leq (\uh(|x|\land \theta_{\uh})/\uK(|x|\land \theta_{\uh}))\, t \uK(|x|),
\end{align*}
since $r\mapsto r^2\uK(r)$ is non-decreasing.
\cite[Lemma~2.3]{TGKS-2019},
continuity of $\uK$ and $\uh$ assert that the quotient $\uh(|x|\land\theta_{\uh})/\uK(|x|\land\theta_{\uh})$ is bounded by a constant depending only on $\ulah, C_{\uh}$.
 Therefore, 
the above minimum is bounded by $c tf(|x|)$.
Now, Corrolary~\ref{cor:por_0} provides 
that
$$
|\partial_{x}^{\bbbeta} p\left(t, x +y+t \drf_{[\uh^{-1}(1/t)]}\right)|\leq c \rr_t\left(x+y+t( \drf_{[\uh^{-1}(1/t)]}- \drf_{\left[1/\LCh^{-1}(1/t))\right]})\right)\,.
$$
The claim follows eventually from 
\eqref{ineq:inverse_equiv}
and Lemma~\ref{lem:small_drif}, and Corollary~\ref{cor:small_shift}.
\qed

\begin{lemma}\label{lem:small_drif}
Assume \eqref{e:assumption-j0}.
Let $r_1,r_2,t>0$ satisfy $c^{-1} \uh^{-1}(1/t)\leq r_i \leq c\uh^{-1}(1/t)$, $i=1,2$,
for a constant $c\geq 1$. Then with $a=c^3 \lmC$ we have 
$$
t |\drf_{r_1}-\drf_{r_2} |\leq a \uh^{-1}(1/t)\,.
$$
\end{lemma}
\pf
We apply \cite[(8)]{TGKS-2019},
\eqref{e:assumption-j0},
$r_1\vee r_2 \leq c \uh^{-1}(1/t)$ and $\uh(r_1\land r_2)\leq \uh(c^{-1}\uh^{-1}(1/t))\leq c^2 t^{-1}$.

\qed

It is straightforward that if $\LM(dz)$ 
 is symmetric, then  $\drf_{[\uh^{-1}(1/t)]}$ reduces to $\drf$ and  \eqref{ineq:Grad_0} simplifies.
Actually, a similar effect may be achieved by a correct choice or replacement of $y$ in the statement of Theorem~\ref{p:upperestonp}.
It is, for instance, feasible locally in time if $\lah>1$, since in such case  
the intrinsic (first order) drift term
is  dominated
by the non-locality.
Another type of reduction is possible if
we can properly control 
$z\LM(dz)$.
We collect those cases in the next result, but first
in a similar fashion to \eqref{eq:wlsc:h}
 we consider
$\beta_{\uh}\in (0,2]$, $c_{\uh}\in (0,1]$ and $\theta_{\uh}\in (0,\infty]$
such that
\begin{equation}\label{eq:wusc:h}
c_{\uh}\,\lambda^{\beta_{\uh}} \,\uh(\lambda r) \leq \uh(r)\, ,\qquad \lambda\leq 1, \,r< \theta_{\uh}\,.
\end{equation}
If $\theta_{\uh}<\infty$ we extend the scaling to $r<R$ as follows, for $\theta_{\uh}\leq r <R$,
$$
h(r)\geq (\theta_{\uh}/R)^2 h(\theta_{\uh}) \geq c_{\uh} (\theta_{\uh}/R)^2 \lambda^{\beta_{\uh}} h(\lambda \theta_{\uh}) \geq [c_{\uh} (\theta_{\uh}/R)^2] \lambda^{\beta_{\uh}} h(\lambda r)\,.
$$

\begin{proposition}\label{prop:shift_change}
Suppose that the assumptions of Theorem~\ref{p:upperestonp} are satisfied.
Let $\bbbeta\in \mathbb{N}_0^d$,
$c_d=16(1+2d)$. 
Then
\begin{align*}
|\partial_x^{\bbbeta}\, p\left(t,x+t\drf_{r_*}\right)|\leq 
\ubC \left[\uh^{-1}(1/t) \right]^{-|\bbbeta|} 
\rr_t(x)\,,
\end{align*}
holds  for all $t< 1/[2 c_d  \lmC \uh(\theta_{\uh})]$ and $x\in\Rd$ in each of the following cases:
\\

\begin{enumerate}
\item[i)] with $r_*=1$ and $\ubC=\ubC(d,\bbbeta,\ulah,C_{\uh},\theta_{\uh},\lmC, \uh)$ if  $\ulah>1$ and $\theta_{\uh}<\infty$ in \eqref{eq:wlsc:h}.\\

\item[ii)] with $r_*\in [0,\infty]$ and $\ubC=\ubC(d,\bbbeta,\ulah,C_{\uh},\lmC,\uh, c_*)$ if $\drf_{r_*}$ is well defined and there is a constant $c_*\geq 0$ such that
for all $r<\theta_{\uh}$,
\begin{align*}
\left| \int_{\Rd} z \left(\ind_{|z|<r} - \ind_{|z|<r_*}\right)\LM(dz) \right| \leq c_* r \uh(r) \,.
\end{align*}
Note: if $r_*\in (0,\infty)$, then $\drf_{r_*}$ is well defined. The marginal cases $\drf_0=\drf-\int_{|z|<1}z\LM(dz)$ and $\drf_{\infty}=\drf+\int_{|z|\geq 1}z \LM(dz)$
require verification.\\

\item[iii)] with $r_*=\theta_{\uh}$ and 
$\ubC=\ubC(d,\bbbeta,\ulah,C_{\uh},\lmC, \uh)$
if $\ulah>1$ in \eqref{eq:wlsc:h}.\\

\item[iv)] with $r_*=0$ and 
$\ubC=\ubC (d,\bbbeta,\ulah,C_{\uh},\beta_{\uh},c_{\uh},\lmC)$
if \eqref{eq:wusc:h} holds 
with $\beta_{\uh}<1$. 
\end{enumerate}
\end{proposition}
\pf
By \eqref{e:assumption-j0} we have $\lmC^{-1} \uh (r)\leq h(r) \leq \lmC \uh (r)$, $r>0$, and \eqref{eq:wlsc:h} transfers to $h$.  For the proof of part i) we apply \cite[Corollary~2.11]{TGKS-2019} 
with $r=\uh^{-1}(1/t)<\theta_{h_0}$ to get for all $t<1/\uh(\theta_{\uh})$,
$$t |\drf- \drf_{[\uh^{-1}(1/t)]} |
\leq a\uh^{-1}(1/t) \,,$$ 
where $a=a(\ulah,C_{\uh},\theta_{\uh},\lmC)$.
The result follows from Theorem~\ref{p:upperestonp} with $y=t(\drf - \drf_r)$ and $\theta=a$.
Now, by the assumptions of part ii)  with $r=\uh^{-1}(1/t)<\theta_{h_0}$ we get
$t|\drf_{r_*}-\drf_r|\leq t c_* rh_0(r)=c_* h_0^{-1}(1/t)$
for all $t<1/\uh(\theta_{\uh})$,
 and the inequality follows from 
Theorem~\ref{p:upperestonp}
with $y=t(\drf_{r_*}-\drf_r)$ and $\theta=c_*$.
Part iii) results from part ii)
by
Lemma~\ref{lem:int_J} (see also \cite[Lemma~2.10]{TGKS-2019})
and it covers the case $\theta_{\uh}=\infty$, since then the integral defining $\drf_{\infty}$ converges absolutely.
Part iv) also follows 
from part ii)
and
Lemma~\ref{lem:int_J}, the quantity $\drf_0$ is well defined. 

\qed

We pass to the lower estimates of $p(t,x)$.
Corollary~\ref{cor:low} follows from
Remark~\ref{rem:uni_scal_C3} and
\cite[Theorem~5.2]{TGKS-2019}.
Corollary~\ref{cor:low_space}
is a consequence of 
Remark~\ref{rem:uni_scal_C3} and
Theorem~\ref{thm:low_C3_sym_Rd}.

\begin{corollary}\label{cor:low}
Assume that the L{\'e}vy measure $\LM(dx)$ satisfies
\eqref{e:assumption-j0} and that \eqref{eq:wlsc:h} holds for~$\uh$.
Then
for all $T, \theta>0$ there 
is a constant
$\tilde{c}=\tilde{c}(d,\ulah,C_{\uh},\theta_{\uh},\uLM,\lmC,T,\theta)>0$ such that
for all $0<t<T$ and $|x|\leq\theta \uh^{-1}(1/t)$,
$$
p(t,x+t \drf_{[ \uh^{-1}(1/t)]}) \geq \tilde{c} \left[ \uh^{-1}(1/t)\right]^{-d}\,.
$$
If $\theta_{\uh}=\infty$, we can also take $T=\infty$ with $\tilde{c}>0$.
\end{corollary}

\begin{corollary}\label{cor:low_space}
Assume that the L{\'e}vy measure $\LM(dx)$ satisfies
\eqref{e:assumption-j0} and that \eqref{eq:wlsc:h} holds for~$\uh$.
Then for all $T,\theta_1, \theta_2 >0$ there 
is a constant
$\tilde{c}=\tilde{c}(d,\ulah,C_{\uh},\theta_{\uh},\uLM,\lmC,T,\theta_1,\theta_2)>0$
such that for all
$0<t<T$,
$|y|\leq \theta_1 \uh^{-1}(1/t)$
and
 $|x|\geq \theta_2 \uh^{-1}(1/t)$,
\begin{align*}
p(t,x+y+t\drf_{[\uh^{-1}(1/t)]}) \geq \tilde{c} \,t \uLM (|x|)\,, 
\end{align*}
If $\theta_{\uh}=\infty$, we can also take $T=\infty$ with $\tilde{c}>0$.
\end{corollary}

\begin{remark}\label{rem:lower_bound_sharp}
If we put Corollary~\ref{cor:low} and 
Corollary~\ref{cor:low_space}
together,
we find out  that 
under \eqref{e:assumption-j0} and~\eqref{eq:wlsc:h}
for
all $0<t<T$,
$|y|\leq \theta_1 \uh^{-1}(1/t)$
and $x\in\Rd$,
\begin{align*}
p(t,x+y+t\drf_{[\uh^{-1}(1/t)]}) \geq \tilde{c} \left( \left[ \uh^{-1}(1/t)\right]^{-d}  \land  \,t \uLM (|x|)\right)\,.
\end{align*}

\noindent
It is natural to ask under which conditions the lower bound
agrees with the upper bound~\eqref{ineq:Grad_0},
i.e., when the estimates are sharp 
(at least locally in time and space).
In our setting 
the question reads whether $\uLM(r)$ is comparable with $r^{-d}\uK(r)$,
that is if the latter is dominated by $\uLM(r)$
(the converse always holds,
Lemma~\ref{lem:basic_prop_K_h+}).
The answer to this question is given in Lemma~\ref{lem:comp_uK_uLM}
by means of scaling properties of $\uK$, or equivalently of $\uLM$.
\end{remark}

Here is a direct consequence of 
Remark~\ref{rem:lower_bound_sharp}
and Lemma~\ref{lem:comp_uK_uLM}
(cf. \cite[Theorem~21 and~26]{MR3165234}).

\begin{corollary}\label{cor:low_sharp}
Assume that the L{\'e}vy measure $\LM(dx)$ satisfies
\eqref{e:assumption-j0} and that \eqref{eq:wlsc:h} holds for~$\uh$.
Furthermore, let $\beta_{\uLM}\in [0,2)$, $c_{\uLM}\in (0,1]$ and $R_{\uLM}\in (0,\infty]$
be such that for all $\lambda\leq 1$ and $r<R_{\uLM}$,
$$
c_{\uLM} \lambda^{d+\beta_{\uLM}} \uLM(\lambda r) \leq \uLM(r)\,.
$$
Then
for every $T>0$
there is 
$\tilde{c}=\tilde{c}(d,\ulah,C_{\uh},\theta_{\uh},\uLM,\lmC,T,\beta_{\uLM},c_{\uLM})$ such that for all 
$0<t<T$ and $|x|<R_{\uLM}$,
$$
p(t,x+t\drf_{[\uh^{-1}(1/t)]}) \geq \tilde{c}\, \rr_t(x) \,.
$$
If $\theta_{\uh}=\infty$, we can also take $T=\infty$ with $\tilde{c}>0$.
\end{corollary}

\noindent
{\bf Proof of Theorem~\ref{thm:int1}}
We use \cite[Remark~2.12]{TGKS-2019} when needed without mentioning.
Note that the condition \eqref{e:assumption-j0} is satisfied with $\uLM=g$.
We show that {\it(a)} implies {\it(c)}. 
By \cite[Lemma~2.3]{TGKS-2019}
and 
$h\approx \uh$
we  have \eqref{eq:wlsc:h}.
Then the first inequality follows from
Theorem~\ref{p:upperestonp} and
the comparability of $K$ and $\uK$, $h^{-1}$ and $\uh^{-1}$.
The second inequality follows from
Remark~\ref{rem:lower_bound_sharp}
with $y=t\drf_{[h^{-1}(1/t)]}-t\drf_{[\uh^{-1}(1/t)]}$.
For both see Lemma~\ref{lem:small_drif} and Remark~\ref{cor:small_shift}.
Obviously, {\it (c)} gives $\Ca$.
Thus {\it (c)} implies {\it (b)} by 
\cite[Theorem~3.1]{TGKS-2019}.
Finally, {\it (b)} implies {\it (a)}
by 
\cite[Lemma~2.4 and 2.3]{TGKS-2019}.
The remaining part follows from
 Lemma~\ref{lem:comp_uK_uLM}.
\qed

\section{Exapmles}\label{sec:exls}

In Examples~\ref{ex:1} -- \ref{ex:2} we consider 
a L{\'e}vy process $Y=(Y_t)_{t\geq 0}$ in $\Rd$
with the generating triplet $(A,\LM,\drf)$, 
where $A=0$, $\drf\in\Rd$ and 
the L{\'e}vy measure $\LM(dx)=\gLM(x)dx$
satisfies for some $0<\lambda\leq \Lambda<\infty$
and all $x\in\Rd$,
$$
\lambda \uLM(|x|) \leq \gLM(x)\leq \Lambda \uLM(|x|)\,.
$$
We emphasise that $\LM(dx)$ may be non-symmetric.
The estimates we obtain are uniform for all $\LM(dx)$
whenever $\uLM$, $\lambda$ and $\Lambda$ are fixed.
First, we provide the Aronson-type estimates.
\begin{example}\label{ex:1}
Let $\uLM(r)=r^{-d-\alpha}+r^{-d-\beta}$, where 
$\alpha,\beta\in(0,2)$.
Then there is $c_0=c_0(d,\alpha,\beta,\lambda,\Lambda)$
such that for all $t>0$ and $x\in\Rd$,
$$
c_0^{-1}
f(t,x)  \leq  p(t,x+t\drf_{[t^{1/\alpha} \vee t^{1/\beta}]}) \leq c_0 f(t,x)\,,
$$
with
$$
f(t,x)= (t^{-d/\alpha}\land t^{-d/\beta})\land
\left( \frac{t}{|x|^{d+\alpha}}+\frac{t}{|x|^{d+\beta}}\right).
$$
Indeed, note that $\uh(r)= c_{d.\alpha} r^{-\alpha}+ c_{d.\beta} r^{-\beta}$
and $\uK(r)=(\alpha/2)c_{d.\alpha} r^{-\alpha}+ (\beta/2)c_{d.\beta} r^{-\beta}$.
Therefore
\eqref{e:assumption-j0},
\eqref{eq:wlsc:h} and~\eqref{eq:wusc:h} are satisfied with $\theta_{\uh}=\infty$,
and
$r^{-d}\uK(r)$ is comparable with $\uLM(r)$.
Since $\uh^{-1}(1/t)$ is comparable with 
$t^{1/\alpha}\vee t^{1/\beta}$, by Lemma~\ref{lem:small_drif} we have that $y=t\drf_{[t^{1/\alpha} \vee t^{1/\beta}]} -t \drf_{[\uh^{-1}(1/t)]}$ satisfies $|y|\leq a \uh^{-1}(1/t)$ with $a=a(d,\alpha,\beta)$.
Thus, 
by Theorem~\ref{p:upperestonp}
and Remark~\ref{rem:lower_bound_sharp}
for all $t>0$ and $x\in\Rd$ we have
$$
c_1^{-1} \rr_t(x)  \leq  p(t,x+y+t\drf_{[\uh^{-1}(1/t)]}) \leq 
c_1\rr_t(x)\,.
$$
Finally,
the comparability of $\rr_t(x)$ and $f(t,x)$ follows from that for
$\uK$ and $\uh$.
\end{example}

\begin{example}\label{ex:1b}
Let $0<\beta < 1 < \alpha<2$
in Example~\ref{ex:1}.
Then there is
$c_0=c_0(d,\alpha,\beta,\lambda,\Lambda)$
such that
for all $t>0$ and $x\in\Rd$,
$$
c_0^{-1}
f(t,x)  \leq  p(t,x+t\drf) \leq c_0 f(t,x)\,,
$$
Indeed,
in that case
$|\drf_r-\drf|\leq c (d,\alpha,\beta) \left( r^{1-\alpha} + r^{1-\beta}\right)$.
Since $r^{\alpha} \uh(r)\geq c_{d.\alpha}$ and $r^{\beta}\uh(r)\geq c_{d.\beta}$ we get
$[\uh^{-1}(1/t)]^{\alpha} t^{-1}\geq c_{d.\alpha}$ and
$[\uh^{-1}(1/t)]^{\beta} t^{-1}\geq c_{d.\beta}$. 
This
gives 
$
t |\drf_{[\uh^{-1}(1/t)]}-\drf|
\leq a \uh^{-1}(1/t)
$
for all $t>0$ with $a=a(d,\alpha,\beta)$,
and we can apply Theorem~\ref{p:upperestonp}
and Remark~\ref{rem:lower_bound_sharp} with $y=t(\drf -\drf_{[\uh^{-1}(1/t)]})$.
\end{example}

Now, we prove sharp estimates for $\uLM$ with 
very heavy tail.

\begin{example}\label{ex:2}
Let $\uLM(r)=r^{-d} [ \log (1+r^{\alpha/2})]^{-2}$, where  $\alpha\in (0,2)$.
Then
for every $T>0$
there is $c_0=c_0(d, \alpha, \lambda,\Lambda,T)$
such that
for all $0<t<T$ and $x\in\Rd$,
$$
c_0^{-1}
f(t,x)  \leq  p(t,x+t\drf_{[\uh^{-1}(1/t)]}) \leq c_0 f(t,x)\,,
$$
with
$$
f(t,x)= t^{-d/\alpha}  \land  \frac{t\, (\log(1+|x|^{\alpha/2}))^{-2}}{|x|^d}\,.
$$
In what follows the comparability constants depend only on $d, \alpha, \lambda,\Lambda,T$.
Since $\uLM(r) \approx r^{-d-\alpha} + r^{-d} (\log r)^{-2} \ind_{r\geq 2}$ we have
$\uh(r)\approx r^{-\alpha}$ for $r<\uh^{-1}(1/(aT))$,
$a=2 c_d (1\vee \lambda^{-1}\vee \Lambda )$.
Thus \eqref{eq:wlsc:h} holds and $\uh^{-1}(1/t)\approx t^{1/\alpha}$ for $t<T$.
We also have $c \lambda^{d+\alpha} \uLM(\lambda r)\leq \uLM(r)$ for all $\lambda \leq 1$, $r>0$.
The upper 
estimate
 follows now from Theorem~\ref{p:upperestonp}
and Lemma~\ref{lem:comp_uK_uLM}.
The lower estimate results from
Corollary~\ref{cor:low_sharp}.
\end{example}

In Examples~\ref{ex:3} -- \ref{ex:5}
we concentrate on
a L{\'e}vy process $Y=(Y_t)_{t\geq 0}$ in $\RR$
with the generating triplet $(A,\LM,\drf)$, 
where $A= 0$, $\drf\in\RR$ and 
the L{\'e}vy measure $\LM(dx)=\gLM(x)dx$
satisfies for all $x>0$,
$$
\gLM(-x) \approx \nu_-(-x)\,, \qquad
\gLM(x) \approx \nu_+(x)\,.
$$

\begin{example}\label{ex:3}
Let $\nu_-(r)=r^{-1-\alpha_-}$ and
$\nu_+(r)=r^{-1-\alpha_+}$, where
$\alpha_-, \alpha_+\in (0,2)$ and 
$\alpha_-\leq \alpha_+$. Then for every $T>0$ there exists a constant $c>0$ such that
for all $0<t<T$,
\begin{align*}
p(t,x+t \drf_{[1/\LCh^{-1}(1/t)]}) &\leq  c t^{-1/\alpha_+}\,,\qquad x\in\RR  \,,\\
p(t,x+t \drf_{[1/\LCh^{-1}(1/t)]}) &\leq  c \frac{ t}{x^{1+\alpha_+}}\,, \qquad x>0 \,,
\end{align*}
\begin{align*}
p(t,-x+t \drf_{[1/\LCh^{-1}(1/t)]})\leq c  \left(\frac{t}{x^{1+\alpha_-}}+ t^{-1/\alpha_+} e^{-c^{-1} t^{-1/\alpha_+} x \log(1+c^{-1} t^{-1/\alpha_+} x)} \right), \quad x>0\,.
\end{align*}
The first two inequalities are equivalent with
$
p(t,x+t \drf_{[1/\LCh^{-1}(1/t)]})\leq  c \min\{t^{-1/\alpha_+}, t x^{-1-\alpha_+}\}$
when restricted to  positive $x$,
and
$tx^{-1-\alpha_+} < t^{-1/\alpha_+}$ holds if and only if $x> t^{1/\alpha_+}$.
The third estimate may be simplified in certain regions. For instance,
$$
p(t,-x+t \drf_{[1/\LCh^{-1}(1/t)]})\leq  c \frac{t}{x^{1+\alpha_-}}\,, \quad\qquad x> t^{1/\alpha_+} \log(1+1/t)\,.
$$
On the other hand, there is $\theta_0>0$ such that for all $0<t<T$ and $x>\theta_0\, t^{1/\alpha_+}$,
\begin{align*}
p(t,-x+t \drf_{[1/\LCh^{-1}(1/t)]})\geq c^{-1} \frac{t}{x^{1+\alpha_-}} \,,
\qquad \qquad
p(t,x+t \drf_{[1/\LCh^{-1}(1/t)]})\geq c^{-1} \frac{t}{x^{1+\alpha_+}}\,.
\end{align*}

\noindent
We justify the estimates.
Note that 
$h(r)\approx r^{-\alpha_+}$, $r\leq 1$,
and thus by
\cite[Remark~3.2]{TGKS-2019}
the condition $\Cc$
 is satisfied with $\alpha_3=\alpha_+$, any fixed $T_3\in(0,\infty)$
and some $c_3\in(0,1]$,
see \cite[Remark~2.12]{TGKS-2019}.
In particular,
\cite[Theorem~3.1]{TGKS-2019}
guarantees $\Ca$  with $T_1=1/h(T_3)$, which is the first estimate.
Since 
by Lemma~\ref{lem:eqiv_h_LCh}
the inequality
\eqref{KP_as8} holds
with $T=1/(2h(T_3))$,
the assumptions of
Proposition~\ref{prop:one_side}
are satisfied with
 $f_+(x)=x^{-1-\alpha_+}$ and $\gamma_+=1$.
Further,
by
Lemma~\ref{lem:inverse_c}
and
\eqref{ineq:comp_TJ_inverse}
we have for all $t\in (0,T)$ and $x>0$,
\begin{align*}
\LCh^{-1}(1/t)x \approx 
t^{-1/\alpha_+} x 
=\left( t^{1+1/\alpha_+} x^{-1-\alpha_+}\right)^{-1/(1+\alpha_+)}
\approx \left( t \left[\LCh^{-1}(1/t)\right]^{-\gamma_+} f_+(x)\right)^{-1/(1+\alpha_+)}\,.
\end{align*}
Thus, if $t^{-1/\alpha_+}x \geq 1$
the exponential term in the conclusion of 
Proposition~\ref{prop:one_side}
is negligible. If $t^{-1/\alpha_+}x \leq 1$
we use the estimate by $t^{-1/\alpha_+}$, which finally gives
$
p(t,x+t \drf_{[1/\LCh^{-1}(1/t)]})\leq c  t^{-1/\alpha_+} \min\{1, t^{1+1/\alpha_+}x^{-1-\alpha_+} \}
$
for all $t\in (0,T)$ and $x>0$. 
Similarly, we use 
Corollary~\ref{cor:one_side} to obtain the third estimate.
Lower estimates follow from 
Proposition~\ref{prop:low_C3_R1}.
Note that $|t \drf_{[1/\LCh^{-1}(1/t)]}-t\drf_{[h^{-1}(1/t)]}| \leq c  h^{-1}(1/t)$ for $t\in (0,T)$, see
\eqref{ineq:comp_TJ_inverse} and \cite[(8)]{TGKS-2019}.
\end{example}

\begin{example}\label{ex:3b}
Let $\alpha_+\geq 1$ in Example~\ref{ex:3}.
Then for every $T >0$ there 
is a constant
$c>0$
such that
for all $0<t<T$, if $|x|\leq t^{1/\alpha_+}$, then
$$
p(t,x+t \drf_{[1/\LCh^{-1}(1/t)]})\geq c^{-1}\,  t^{-1/\alpha_+}\,,
$$
and if
$x\geq  t^{1/\alpha_+}$, then
\begin{align*}
p(t,-x+t \drf_{[1/\LCh^{-1}(1/t)]})\geq c^{-1} \frac{t}{x^{1+\alpha_-}} \,,
\qquad \qquad
p(t,x+t \drf_{[1/\LCh^{-1}(1/t)]})\geq c^{-1} \frac{t}{x^{1+\alpha_+}}\,.
\end{align*}

\noindent
Indeed, the first inequality follows from \cite[Theorem~5.3]{TGKS-2019}
with
$x$ replaced by $x+ t \drf_{[1/\LCh^{-1}(1/t)]}-t\drf_{[h^{-1}(1/t)]}$.
The next two are consequences of Corollary~\ref{cor:low_C3_scl1_R1}
 with $y= t \drf_{[1/\LCh^{-1}(1/t)]}-t\drf_{[h^{-1}(1/t)]}$.

\begin{table}[h!]
\begin{center}
\caption{
Estimates of $p(t,x+t \drf_{[1/\LCh^{-1}(1/t)]})$
for $0<t<T$ in Example~\ref{ex:3b};}
\begin{tabular}[t]{|c|c|c|c|c|}
\hline
 \makecell{range\\ of variables} & $x<-t^{1/\alpha_+} L(t)$  & $-t^{1/\alpha_+} L(t)<x<-t^{1/\alpha_+}$ & $|x|<t^{1/\alpha_+}$  & $x>t^{1/\alpha_+}$ \\ 
\hline\hline
\makecell{upper\\bound} & \multirow{2}{1cm}{$\dfrac{t}{|x|^{1+\alpha_-}}$} &
$\dfrac{t}{|x|^{1+\alpha_-}}+t^{-1/\alpha_+} \cdot e^{-\Theta(t,x)}$  &  \multirow{2}{1cm}{$t^{-1/\alpha_+}$} & \multirow{2}{1,5cm}{$ \dfrac{t}{|x|^{1+\alpha_+}}$} \\  
\cline{1-1} \cline{3-3}
 \makecell{lower\\bound}&  & $\dfrac{t}{|x|^{1+\alpha_-}}$ &  &  \\  
\hline 
\multicolumn{5}{l}{}\\
\multicolumn{5}{l}{\footnotesize{$L(t)=\log(1+1/t)$};\qquad
\footnotesize{$\Theta(t,x)= {c^{-1} t^{-1/\alpha_+} |x| \log(1+c^{-1} t^{-1/\alpha_+} |x|)}$;}}
\end{tabular}
\end{center}
\end{table}
\end{example}

\begin{example}\label{ex:4}
Let 
$\nu_-(r)=r^{-1-\alpha_-}e^{-c_-\,r }$
and
$\nu_+(r)=r^{-1-\alpha_+} e^{- c_+\,r}$, where 
$0<\alpha_-\leq \alpha_+<2$
and $c_-,c_+\geq 0$.
Then for every $T>0$ there exists a constant $c>0$ 
such that
for all $0<t<T$
all the upper estimates 
in Example~\ref{ex:3} are valid in the present example. Further, there is $x_0>0$ such that for all $0<t<T$ and $x>x_0$,
\begin{align*}
p(t,-x+t \drf_{[1/\LCh^{-1}(1/t)]})\leq  c \frac{t}{x^{1+\alpha_-}}\, e^{-c_-\,x}\,,
\qquad \quad
p(t,x+t \drf_{[1/\LCh^{-1}(1/t)]})
\leq  c \frac{t}{x^{1+\alpha_+}}\, e^{-c_+\,x}\,.
\end{align*}
On the other hand, there is $\theta_0>0$ such that for all $0<t<T$ and $x>\theta_0\, t^{1/\alpha_+}$,
\begin{align*}
p(t,-x+t \drf_{[1/\LCh^{-1}(1/t)]})\geq c^{-1} \frac{t}{x^{1+\alpha_-}} e^{-c_- x} \,,
\qquad \quad
p(t,x+t \drf_{[1/\LCh^{-1}(1/t)]})\geq c^{-1} \frac{t}{x^{1+\alpha_+}}e^{-c_+ x} \,.
\end{align*}

\noindent
We justify the estimates.
Since $h(r)\approx r^{-\alpha_+}$ for $r\leq 1$,
and $\nu_-(r)\leq r^{-1-\alpha_-}$, $\nu_+(r)\leq r^{-1-\alpha_+}$ for $r>0$, the upper estimates obtained in Example~\ref{ex:3} are valid in the current situation
by Proposition~\ref{prop:one_side}.
Refined upper bounds for $x>x_0$ follow from Theorem~\ref{thm:up_d1} and Remark~\ref{rem:x_neg1}, see Lemma~\ref{lem:f_aux}a).
Here we also use a general fact that
$t \drf_{[1/\LCh^{-1}(1/t)]}$
is bounded for $t\in (0,T)$.
Lower estimates follow from
Proposition~\ref{prop:low_C3_R1}.
Note that $|t \drf_{[1/\LCh^{-1}(1/t)]}-t\drf_{[h^{-1}(1/t)]}| \leq c  h^{-1}(1/t)$ for $t\in (0,T)$, see
\eqref{ineq:comp_TJ_inverse} and \cite[(8)]{TGKS-2019}.
\end{example}

\begin{example}\label{ex:4b}
Let $\alpha_+\geq 1$ in Example~\ref{ex:4}.
Then for every $T >0$ there 
is a constant
$c>0$
such that
for all $0<t<T$, if $|x|\leq t^{1/\alpha_+}$, then
$$
p(t,x+t \drf_{[1/\LCh^{-1}(1/t)]})\geq c^{-1}\,  t^{-1/\alpha_+}\,,
$$
and if
$x\geq  t^{1/\alpha_+}$, then
\begin{align*}
p(t,-x+t \drf_{[1/\LCh^{-1}(1/t)]})\geq c^{-1} \frac{t}{x^{1+\alpha_-}}e^{-c_- x} \,,
\qquad \quad
p(t,x+t \drf_{[1/\LCh^{-1}(1/t)]})\geq c^{-1} \frac{t}{x^{1+\alpha_+}}e^{-c_+ x}\,.
\end{align*}

\noindent
Indeed, the first inequality follows from \cite[Theorem~5.3]{TGKS-2019}
with
$x$ replaced by $x+ t \drf_{[1/\LCh^{-1}(1/t)]}-t\drf_{[h^{-1}(1/t)]}$.
The next two are consequences of Corollary~\ref{cor:low_C3_scl1_R1}
 with $y= t \drf_{[1/\LCh^{-1}(1/t)]}-t\drf_{[h^{-1}(1/t)]}$,
 see Remark~\ref{rem:x_neg2}.

\begin{table}[h!]
\begin{center}
\caption{
Estimates of $p(t,x+t \drf_{[1/\LCh^{-1}(1/t)]})$
for $0<t<T$ in Example~\ref{ex:4b};}
\begin{tabular}[t]{|c|c|c|c|c|}
\hline
 \makecell{range\\ of variables} & $x<-t^{1/\alpha_+} L(t)$  & $-t^{1/\alpha_+} L(t)<x<-t^{1/\alpha_+}$ & $|x|<t^{1/\alpha_+}$  & $x>t^{1/\alpha_+}$ \\ 
\hline\hline
\makecell{upper\\bound} & \multirow{2}{1cm}{$\dfrac{t e^{-c_-|x|}}{|x|^{1+\alpha_-}}$} &
$\dfrac{t}{|x|^{1+\alpha_-}}+t^{-1/\alpha_+} \cdot e^{-\Theta(t,x)}$  &  \multirow{2}{1cm}{$t^{-1/\alpha_+}$} & \multirow{2}{1,5cm}{$ \dfrac{te^{-c_+|x|}}{|x|^{1+\alpha_+}}$} \\  
\cline{1-1} \cline{3-3}
 \makecell{lower\\bound}&  & $\dfrac{t}{|x|^{1+\alpha_-}}$ &  &  \\  
\hline 
\multicolumn{5}{l}{}\\
\multicolumn{5}{l}{\footnotesize{$L(t)=\log(1+1/t)$};\qquad
\footnotesize{$\Theta(t,x)= {c^{-1} t^{-1/\alpha_+} |x| \log(1+c^{-1} t^{-1/\alpha_+} |x|)}$;}}
\end{tabular}
\end{center}
\end{table}
\end{example}

We note that we can take $\alpha_-=\alpha_+$ in Examples~\ref{ex:3} - \ref{ex:4b}, and then the exponential term including $\Theta(t,x)$ is redundant. What is more, in such case we can improve lower bounds obtained in Example~\ref{ex:4}, by applying 
\cite[Theorem~5.3]{TGKS-2019} and Proposition~\ref{prop:R1}, to get result like in Example~\ref{ex:4b}, but without any additional restriction on $\alpha_-=\alpha_+$.
\begin{example}\label{ex:5}
Let 
$\nu_-(r)=r^{-1-\alpha} e^{-c_-\,r }$
and
$\nu_+(r)=r^{-1-\alpha} e^{- c_+\,r}$, where 
$\alpha\in(0,2)$
and $c_-,c_+\geq 0$.
Then all estimates in Example~\ref{ex:4} and~\ref{ex:4b} with $\alpha_-=\alpha_+=\alpha$ are valid in the present setting.

We only need to justify lower bounds of Example~\ref{ex:4b}. To this end we use
\cite[Theorem~5.2]{TGKS-2019} with 
$x$ replaced by $x+ t \drf_{[1/\LCh^{-1}(1/t)]}-t\drf_{[h^{-1}(1/t)]}$,
and Proposition~\ref{prop:R1} (variant of Theorem~\ref{thm:low_C3_sym_Rd})
with
 $y= t \drf_{[1/\LCh^{-1}(1/t)]}-t\drf_{[h^{-1}(1/t)]}$.
In both results we take
 $\nu_s(dx)=|x|^{-1-\alpha} e^{-\max\{c_-,c_+\} |x|}dx$,
since then 
$\nu_s(dx)\leq c \LM(dx)$
and
for $|x|\geq 1$ we have
$
{\rm Re}[\LCh (x)]
\approx |x|^{\alpha} \approx {\rm Re}[\LCh_s (x)]
$.
See
\cite[Remark~3.2]{TGKS-2019}
and Remark~\ref{rem:x_neg2}.
\begin{table}[h!]
\begin{center}
\caption{
Estimates of $p(t,x+t \drf_{[1/\LCh^{-1}(1/t)]})$ for $0<t<T$ in Example~\ref{ex:5};}
\begin{tabular}[t]{|c|c|c|c|}
\hline
 \makecell{range\\ of variables} & $x<-t^{1/\alpha} $  &   $|x|<t^{1/\alpha}$  & $x>t^{1/\alpha}$ \\ 
\hline\hline
\makecell{upper\\bound} & \multirow{2}{1cm}{$\dfrac{t e^{-c_-|x|}}{|x|^{1+\alpha}}$} 
 &  \multirow{2}{1cm}{$t^{-1/\alpha}$} & \multirow{2}{1,5cm}{$ \dfrac{te^{-c_+|x|}}{|x|^{1+\alpha}}$} \\  
\cline{1-1}
 \makecell{lower\\bound}&  &   &  \\  
\hline 
\end{tabular}
\end{center}
\end{table}
\end{example}

\section{Appendix - unimodal L{\'e}vy processes}\label{sec:uni_app}

In this section
we
discuss
a L{\'e}vy process $X$
with a generating triplet $(0,\uLM,0)$, where 
$\uLM(dx)=\uLM(|x|)dx$
for a non-increasing (profile) function $\uLM\colon [0,\infty)\to [0,\infty]$.
Equivalently,
$X$ is
a pure-jump
isotropic unimodal L{\'e}vy process (see \cite{MR3165234}, \cite{MR705619}).
The characteristic exponent $\uLCh$ of $X$ 
takes the form 
$$\uLCh(x)={\rm Re} [\uLCh(x)]=\int_{\Rd}\big( 1-\cos\left<x,z\right> \big)\uLM(|z|)dz\,,$$
and
satisfies (\cite[Proposition~2]{MR3165234})
\begin{align}\label{ineq:comp_unimod}
(1/\pi^2) \uLCh^*(|x|) \leq \uLCh(x) \leq \uLCh^*(|x|)\,,\qquad x\in\Rd\,.
\end{align}
For $r>0$ we define 
$$
\uh(r)
=\int_{\Rd}\left(1\wedge\frac{|x|^2}{r^2}\right)\uLM(|x|)dx\,,
\qquad
\uK(r)=r^{-2}\int_{|x|<r}|x|^2\uLM(|x|)dx\,.
$$ 
In what follows {\bf we assume that $h_0(0^+)=\uLM(\Rd)=\infty$}, i.e., $X$ is not a compound Poisson process.
First we complement the properties
of $\uK$ and $\uh$
gathered in 
\cite[Section~2]{TGKS-2019}.

\begin{lemma}\label{lem:basic_prop_K_h+}
We have
\begin{enumerate}
\item[\rm 1.] $\uK$ is continuous,
\item[\rm 2.] $r^{-d}\uK(r)$ is strictly decreasing,
\item[\rm 3.] $\uK(r)\leq \lambda^{-d}K_0(\lambda r)$, $\lambda\leq 1$, $r>0$,
\item[\rm 4.] $\uLM(r)\leq (\omega_d/(d+2))^{-1}\, r^{-d}\uK(r)$.
\end{enumerate}
\end{lemma}

 \begin{lemma}\label{lem:int_J}
Let $\uh$ satisfy \eqref{eq:wlsc:h} with $\ulah>1$, then
\begin{align*}
\int_{r \leq |z|<  \theta_{\uh} }  |z|\uLM(|z|)dz \leq \frac{(d+2) C_{\uh}}{\ulah-1} \, r \uh(r)\,, \qquad r>0\,.
\end{align*}
Let $\uh$ satisfy \eqref{eq:wusc:h} with $\beta_{\uh}<1$, then
\begin{align*}
\int_{|z|< r} |z| \uLM(|z|)dz\leq \frac{d+2}{c_{\uh}(1-\beta_{\uh})}\, r \uh(r)\,,\qquad r<\theta_{\uh}\,.
\end{align*}
\end{lemma}
\begin{proof}
Proofs 
are based on the properties of $\uLM$, $\uK$, $\uh$
and on scaling assumptions.
\end{proof}

The equivalence
of conditions presented in the next lemma
 can be derived from
the observations collected in
\cite[Appendix A]{Grzywny2017},
which rely on results of \cite{MR0466438}.
We give a direct short proof and we
mark down relations between parameters.

\begin{lemma}\label{lem:comp_uK_uLM}
The following are equivalent.
\begin{enumerate}
\item[$\Aaa$] There are $T_1\in (0,\infty]$, $c_1>0 $ such that for all $r<T_1$,
\begin{align*}
c_1 r^{-d}\uK(r) \leq  \uLM(r)\,.
\end{align*}
\item[$\Abb$] There are $T_2\in(0,\infty]$, $c_2\in (0,1]$ and $\beta_2\in (0,2)$ such that
for all $\lambda \leq 1$ and $r<T_2$,
\begin{align*}
c_2 \lambda^{\beta_2} \uK(\lambda r) \leq \uK(r)\,.
\end{align*}
\item[$\Acc$] There are $T_3\in (0,\infty]$, $c_3\in (0,1]$ and $\beta_3 \in [0,2)$ such that
for all $\lambda\leq 1$ and $r<T_3$,
$$
c_3 \lambda^{d+\beta_3} \uLM(\lambda r) \leq \uLM(r)\,.
$$
\end{enumerate}
Moreover, $\Aaa$ implies $\Abb$ with
$T_2=T_1$, $c_2=1$ and $\beta_2=\beta_2(d,c_1)$.
From $\Aaa$
we get $\Acc$ with
$T_3=T_1$, $c_3=c_3(d,c_1)$ and $\beta_3=\beta_3(d,c_1)$.
The condition $\Abb$ gives $\Aaa$ with $T_1=(c_2/2)^{1/(2-\beta_2)} T_2$ and $c_1=c_1(d,c_2,\beta_2)$. From $\Acc$ we have $\Aaa$ with $T_1=T_3$ and $c_1=c_1(d,c_3,\beta_3)$.
\end{lemma}
\pf
Define $f(r)= r^2\uK(r)$, $r>0$. For all $0\leq a<b<\infty$ we have 
$$
f(b)-f(a)= \omega_d \int_a^b r^{d+2} \uLM(r)\frac{dr}{r}\,.
$$
If $\Aaa$
holds, then we get for $0<r_1<r_2<T_1$,
$$
f(r_2)-f(r_1)\geq c_1 \omega_d \int_{r_1}^{r_2} f(s)s^{-1}ds\,,
$$
which implies (since $f$ is continuous and positive) that $f(r)r^{-c_1 \omega_d}$ is non-decreasing for $r<T_1$. Thus 
$\Abb$
holds, as well as $\Acc$.
Assume $\Abb$. Then $c_2 f(\lambda r)\leq \lambda^{2-\beta_2} f(r)$, and therefore
 $f(\lambda_0 r)\leq (1/2) f(r)$ for $r<T_2$ and $\lambda_0=(c_2/2)^{1/(2-\beta_2)}$.
Further, for $r<T_2$,
\begin{align*}
\frac{f(\lambda_0 r)}{2}\leq \frac{f(r)}{2} \leq f(r)-f(\lambda_0 r)
=\omega_d \int_{\lambda_0 r}^r s^{d+1}\uLM(s)ds
\leq \omega_d \left( \frac{\lambda_0^{-d-2}-1}{d+2}\right)   (\lambda_0 r)^{d+2} \uLM(\lambda_0 r) \,.
\end{align*}
That gives $c_1 (\lambda_0r)^{-d} \uK(\lambda_0 r) \leq \uLM(\lambda_0 r)$ for $r<T_2$
and proves $\Aaa$.
Assume $\Acc$. Then for $r<T_3$,
\begin{align*}
r^2 \uK(r)=\omega_d \int_0^r s^{d+1} \uLM(s) ds
\leq (\omega_d/c_3) \int_0^r s^{d+1} (s/r)^{-d-\beta_3} \uLM(r) ds
= \frac{\omega_d}{c_3(2-\beta_3)} r^{d+2} \uLM(r)\,.
\end{align*}
This guarantees $\Aaa$ and ends the proof.
\qed

We 
demonstrate
fundamental
 properties of the bound function~\eqref{def:bound_function}.

\begin{lemma}\label{lem:integr_rr}
We have
$$(\omega_d/2)\leq \int_{\Rd}\rr_t(x)\,dx \leq (\omega_d/2)(1+2/d),\qquad t>0\,.$$
\end{lemma}
\pf
Note that
$$\int_{|x|\leq \uh^{-1}(1/t)}\big[\uh^{-1}(1/t) \big]^{-d} dx=\frac{\omega_d}{d}\,,$$
and by \cite[Lemma~2.2]{TGKS-2019},
$$
\int_{|x|>\uh^{-1}(1/t)}\uK(|x|)|x|^{-d}dx
=\frac{\omega_d}{2}\, \frac{1}{t}.
$$
The lower bound is a consequence of  
$\uK(r)\leq \uh(r)$, which implies
$t \uK(|x_0|)|x_0|^{-d}\leq [\uh^{-1}(1/t)]^{-d}$ for $|x_0|= \uh^{-1}(1/t)$,
and the monotonicity of  $r^{-d} \uK(r)$.
\qed

\begin{lemma}\label{lem:sol:rho}
Fix $t>0$. There is a unique solution $r_0>0$ of
$$t\uK(r)r^{-d}=  [\uh^{-1}(1/t)]^{-d}=\rr_t(r)\,,$$ 
and  $r_0\in [\uh^{-1}(3/t),\uh^{-1}(1/t)]$.
\end{lemma}
\pf
Since $\uLM$ is unbounded we get
$\lim_{r\to 0^+}\uK(r)r^{-d}=\infty$. Moreover, $\lim_{r\to\infty}\uK(r)r^{-d}=0$. Thus by continuity and monotonicity of 
$\uK(r)r^{-d}$
 there exists a unique solution.
 Next, since $\uK(r)\leq \uh(r)$, we have 
$r_0\leq \uh^{-1}(1/t)$. To prove the lower bound we use Lemma \ref{lem:integr_rr} to obtain
\begin{align*}
\frac{\omega _d}{2} \uh(r_0) t &=\int_{|z|\geq r_0}\frac{t \,\uK(|z|)}{|z|^d}dz\leq\int_{\Rd} \rr_t(z) dz  \leq  \frac{\omega_d}{2} \frac{d+2}{d}\,,
\end{align*}
which gives $\uh(r_0)\leq (1+2/d)/t\leq 3/t$. Therefore $r_0\geq \uh^{-1}(3/t)$. 
\qed

\begin{proposition}\label{prop:small_shift}
Let $a\geq 1$. There is $c=c(d,a)$ such that
for all $t>0$,
\begin{align*}
\rr_t(x+z)\leq c\, \rr_t(x)\,, \qquad\qquad \mbox{if} \qquad |z|\leq \left[a\,\uh^{-1}(3/t)\right] \vee \frac{|x|}{2}\,.
\end{align*}
\end{proposition}
\pf
If $|z|\leq |x|/2$, then $|x+z|\geq |x|/2$. 
By monotonicity of $\uK(r)r^{-d}$
and scaling of $\uK$,
\begin{align*}
\rr_t(x+z)\leq 
\left[\uh^{-1}(1/t)\right]^{-d}\land \big(t \uK(|x|/2)(|x|/2)^{-d} \big)
\leq 2^{d+2} \rr_t(x)\,.
\end{align*}
Now we prove the second part.  Let $a=1$.
The condition $|z|\leq \uh^{-1}(3/t)$ implies 
$\rr_t(x+z) \leq \rr_t(z)=[ \uh^{-1}(1/t)]^{-d}$
(see Lemma~\ref{lem:sol:rho}).
Again, monotonicity of $\uK(r)r^{-d}$ and scaling of $\uK$ imply for  $2|z|>|x|$
that
$\rr_t(z)\leq 2^{d+2}\rr_t(2z) \leq 2^{d+2} \rr_t(x)$.
Hence for $x\in\Rd$ and $|z|\leq h^{-1}(3/t)$ we obtain
$$\rr_t(x+z)\leq 2^{d+2}\rr_t(x)\,.$$
Finally, to cover the case $a> 1$ we let $n\in\mathbb{N}$. Then for any $|z|\leq n \uh^{-1}(3/t)$ we have
$$
\rr_t(x+z)=\rr_t(x+[(n-1)/n ]z+ z/n) \leq (2^{d+2})^n \rr_t(x)\,.
$$
\qed

\begin{corollary}\label{cor:small_shift}
Let $\uh$ satisfy \eqref{eq:wlsc:h}. For every $a\geq 1$ 
there is $c=c(d,a,\ulah,C_{\uh})$ such that
\begin{align*}
\rr_t(x+z)\leq c\, \rr_t(x)\,, \qquad\qquad \mbox{if} \qquad |z|\leq \left[a\,\uh^{-1}(1/t)\right] \vee \frac{|x|}{2}\quad \mbox{and}\quad t<1/\uh(\theta_{\uh})\,.
\end{align*}
\end{corollary}

\begin{corollary}\label{cor:por_0}
Let $\uh$ satisfy \eqref{eq:wlsc:h}. For $t>0$, $x\in\Rd$ define 
$$
\varphi_t(x)=
\begin{cases}
[\uh^{-1}(1/t)]^{-d},\qquad &|x|\leq \uh^{-1}(1/t)\,,\\
t\uK(|x|)|x|^{-d}, & |x|>  \uh^{-1}(1/t)\,.
\end{cases}
$$
Then $\rr_t(x)\leq \varphi_t(x) \leq c\, \rr_t(x)$ for all $t<1/\uh(\theta_{\uh})$, $x\in\Rd$ and a constant $c=c(\ulah,C_{\uh})$.
\end{corollary}
\pf
By Lemma~\ref{lem:sol:rho} it suffices to consider $r_0<|x|\leq \uh^{-1}(1/t)$. Then by 
\cite[Lemma~2.3]{TGKS-2019},
$$
\varphi_t(x)= [\uh^{-1}(1/t)]^{-d}
= t \uh(\uh^{-1}(1/t)) [\uh^{-1}(1/t)]^{-d} 
\leq 
c\, t \uK(|x|)|x|^{-d}
=
\rr_t(x)\,.
$$
\qed

\section{Appendix - equivalent conditions}
\label{sec:aux}

Let $Y$ be a L{\'e}vy process in $\Rd$
with a generating triplet $(A,\LM,\drf)$.
We assume $h(0^+)=\infty$.

\begin{lemma}\label{lem:eqiv_h_LCh}
The following are equivalent.
\begin{itemize}
\item[\Nb] There are $\hat{T}_1 \in (0,\infty]$, $\hat{c}_1>0$ such that for all $t<\hat{T}_1$,
$$
\int_{\Rd} e^{-t \, {\rm Re}[\LCh(z)]} \,dz \leq \hat{c}_1 \left[\LCh^{-1}(1/t)\right]^{d}\,.
$$
\item[\Nc] There is $\hat{T}_2\in(0,\infty]$ such that
for some (every) $m\in\N$ there is $\hat{c}_2>0$ and for all $t<\hat{T}_2$,
$$
\int_{\Rd}|z|^m e^{-t \, {\rm Re}[\LCh(z)]}  \,dz \leq \hat{c}_2 \left[\LCh^{-1}(1/t)\right]^{d+m}\,.
$$
\item[\Na] There are $T_2\in (0,\infty]$, $c_2>0$ such that for all $t<T_2$,
$$
\int_{\Rd} e^{-t\, {\rm Re}[\LCh(z)]} dz \leq c_2 \left[h^{-1}(1/t)\right]^{-d}.
$$
\item[\Nd]
There are $T_3\in (0,\infty]$, $c_3\in (0,1]$ and $\alpha_3\in (0,2]$ such that for all $|x|>1/T_3$,

$$c_3\,  \LCh^*(|x|)\leq {\rm Re}[\LCh(x)]\qquad \mbox{and}
\qquad 
\LCh^*(\lambda r)\geq c_3 \lambda^{\alpha_3} \LCh^*(r)\,,\quad \lambda \geq 1,\,r>1/T_3\,.
$$ 
\end{itemize}
\end{lemma}

\pf
Combining \eqref{ineq:comp_TJ} with Lemma~\ref{lem:inverse_c} we have
with $c_d=16(1+2d)$,
\begin{align}\label{ineq:comp_TJ_inverse}
\frac{1}{h^{-1}(u/2)} \leq \LCh^{-1}(u)\leq \frac{1}{h^{-1}((c_d/2) u)}\,,\qquad u>0\,.
\end{align}
$\Nb \implies  \Na$.
Using \eqref{ineq:comp_TJ_inverse} and
\cite[Lemma~2.4 and~2.3]{TGKS-2019}
we obtain \cite[(A2)]{TGKS-2019}
 with $\lah=\lah(d,\hat{c}_1)$, $C_h=C_h(d,\hat{c}_1)$ 
 and $\theta_h=h^{-1}(c_d/(2\hat{T}_1))$,
 and leads to $\Cb$
with $c_2=c_2(d,\hat{c}_1)$ and $T_2=2\hat{T}_1/c_d$.

\noindent
$\Na \implies  \Nc$. Again,
\cite[Lemma~2.4 and~2.3]{TGKS-2019}
guarantee 
\cite[(A2)]{TGKS-2019}
 with $\lah=\lah(d,c_2)$, $C_h=C_h(d,c_2)$ 
 and $\theta_h=h^{-1}(1/T_2)$.
 Therefore, $h^{-1}(1/t)\geq (2C_h)^{-1/\lah} h^{-1}(1/(2t))$ for $t<T_2/2$.
Given $m\in\N$ we apply
\cite[Theorem~3.1 and Proposition~3.6]{TGKS-2019},
the latter inequality and 
\eqref{ineq:comp_TJ_inverse}
to reach $\Nc$
with $\hat{c}_2=\hat{c}_2(d,m,c_2)$ and $\hat{T}_2=T_2/2$.

\noindent
$\Nc \implies  \Nb$.
We have
\begin{align*}
\int_{\Rd} e^{-t\, {\rm Re}[\LCh(z)]}\, dz
\leq \int_{|z|\leq \LCh^{-1}(1/t)}dz+  \left[\LCh^{-1}(1/t)\right]^{-m} \int_{|z|>\LCh^{-1}(1/t)}|z|^m e^{-t\, {\rm Re}[\LCh(z)]}\, dz,
\end{align*}
which gives $\Nb$ with $\hat{c}_1=\hat{c}_1(d,\hat{c}_2)$ and $\hat{T}_1=\hat{T}_2$.

\noindent
$\Na \iff \Nd$ It  was shown in \cite[Theorem~3.1]{TGKS-2019}.

\qed

\begin{lemma}\label{lem:inverse_c}
Let $\phi,\varphi:[0,\infty)\mapsto [0,\infty)$ be non-decreasing, continuous, $\phi(0)=\varphi(0)=0$ and $\lim_{r\to \infty}\phi(r)=\lim_{r\to \infty}\varphi(r) =\infty$.
Define (similarly for $\varphi$)
\begin{align*}
\phi^{-1}(s) &= \sup \{r > 0 \colon \phi (r)= s\}
=\sup \{r> 0 \colon \phi (r)\leq  s\},\quad s>0.
\end{align*}
Provided $c>0$,
if $c^{-1}\phi (r)\leq \varphi(r)$ for $r>0$, then $\varphi^{-1}(s)\leq \phi^{-1}(cs)$ for $s>0$. 
\end{lemma}

\section*{Acknowledgment}
The authors thank A. Bendikov, K. Bogdan, A.~Grigor'yan, S. Molchanov and  R. Schilling and P. Sztonyk
for helpful comments.

\vspace{.1in}


\small
\bibliographystyle{abbrv}

\end{document}